\documentclass[11pt,twoside]{article}  
\usepackage{hyperref}
\usepackage[ansinew]{inputenc}
\usepackage{amsfonts, color}
\usepackage{amssymb,amsmath, amsthm}
\usepackage{latexsym}
\usepackage{graphicx}

\usepackage{bm}

\def\E{\mathbb{E}}

\def\P{\mathbb{P}}

\newcommand {\debeq}{\begin{eqnarray*}}
\newcommand {\fineq}{\end{eqnarray*}}

\newcommand {\mbb} {\mathbb}

\def\e{{\rm e}\,}

\newtheorem	{thm}		{Theorem}[section]

\newtheorem	{lem} 	[thm]	{Lemma}
\newtheorem     {rem}      [thm]{Remark}
\newtheorem	{prop}	[thm]{Proposition}
\newtheorem	{cor}		[thm]{Corollary}

\begin{document}

\title{The topology and inference for multi-type Yule trees 
}
\author{Lea Popovic\\ and \\Mariolys Rivas\\ Department of Mathematics and Statistics\\ Concordia University\\ Montreal QC H3G 1M8,  Canada\\
}
\date{\today}
\maketitle
\begin{abstract}
\noindent We introduce two models for multi-type random trees motivated by studies of trait dependence in the evolution of species. Our discrete time model, the {\it multi-type ERM tree}, is a generalization of Markov propagation models on a random tree generated by a binary search or `equal rates Markov' mechanism. Our continuous time model, the {\it multi-type Yule tree with mutations}, is a multi-type generalization of the tree generated by a pure birth or Yule process. 
We study type dependent topological properties of these two random tree models. We derive asymptotic results that allow one to infer model parameters from data on types at the leaves and at branch-points that are one step away from the leaves.
\end{abstract}

\tableofcontents

\noindent \textit{Running head.} The multi-type Yule process.\\
\textit{MSC Subject Classification (2000).} 
\medskip

\section{Introduction}
During the past decade there has been considerable activity in studying the effect trait differences may have on the rates of speciation and extinction in the evolution of species  (\cite{FitzT} gives an excellent presentation of these recent developments). The possibility that diversification may be trait dependent implies that these rates should not be inferred using standard trait independent methods. New likelihood methods that make better use of phylogenetic information were recently developed: ``BiSSE'' for binary state speciation and extinction, \cite{Maddison07}; ``QuaSSE'' for quantitative traits \cite{Fitzjohn10}; ``GeoSSE'' for geographic character traits \cite{Goldberg11}, ``CLASSE'' for punctuated modes of character change \cite{GoldIgic12}; and were used to make new conclusions about a number of different clades (see \cite{NGSmith} for a recent survey).

Inferring the evolutionary process poses in general a non-trivial reconstruction problem, as neither the rates nor the ancestral states in the phylogeny of present day species are known. The underlying ancestral trees are typically assumed to be known and reconstructed from aligned DNA sequence data.  Predicting ancestral states is then typically done with one of a number of heuristic methods based on the principle of either: counting, maximum parsimony, or maximum likelihood. In such studies a Markov chain of state changes is assumed to propagate down from the root along the given tree. Many interesting theoretical results exist on the ability to reconstruct ancestral states along the tree and the state at the root  from the states observed at the leaves (see \cite{MossSteel1} for a survey and \cite{GascSteel} for some recent developments) for which results from statistical physics theory have been  particularly useful (\cite{MosselS1,MosselS2}). The focus so far was on reconstructing hidden states along the underlying tree, rather than parameters of the Markov chain which propagates them. In phylogenetics the underlying tree is assumed to be either a random discrete binary tree, or a random Yule tree generated by a neutral pure birth process. The shape of such a tree has the distribution of `equal rates Markov' (ERM), and the tree resulting after propagating types can be called a {\it multi-type ERM tree}. We extend the model of propagating types down the tree to also include correlations of types between edges with the same branch-point. Consequently, information on leaves will be insufficient for reconstruction and we will also use correlated substructures of the tree. 

Inference for evolutionary processes whose birth and death rates are trait dependent adds an additional layer of mathematical difficulty. If we have a trait with finitely many variants (or a continuum of variants is discretized into finitely many bins) the full tree evolves according to a multi-type birth-death process, in which rates of speciation to different offspring types and the rate of extinction is specific to the type of that lineage. In such a branching process the shape of the tree and its edge lengths are inseparable from the distribution of states on the lineages. Both the ratio of speciation to extinction rates for each state, as well as the transitions from a certain state to another play an important role in how the ancestral states are distributed along the tree. The ancestral tree of this branching process, obtained by pruning away the extinct lineages, turns out to be a random tree we call {\it multi-type Yule tree with mutations}. In such a tree the chance of a lineage splitting is state dependent, which leaves a signature on which splits are more frequent than others and is reflected in the proportion of different types of splits at the tips of the tree. We will use this information in the reconstruction of model parameters.

Very few theoretical results have been obtained for ancestral trees of multi-type branching processes. Deriving an exact distribution for the ancestral tree is unsurprisingly challenging, as determining the likelihood of any split requires the knowledge of the parental type, and hence also all ancestral states on that lineage. In \cite{PopRivas} we developed a coalescent point-process approach to generating ancestral trees using the tips in an infinite (quasi-stationary) multi-type Galton-Watson branching processes. This construction relied on a horizontal exploration of the tips which was developed in \cite{AldPopovic} and extended in \cite{LambPopovic} (the standard vertical coalescent construction is not possible for branching processes with type-dependent offspring distributions). It  can be used for simulating and computing likelihood of ancestral trees, but calculating its statistical features is not easy, except in some very special cases. 

Instead we focus on analyzing newly introduced a priori models on possible ancestral tree shapes. Our {\it multi-type ERM tree} is a discrete time model that is a extension of Markov propagation models on a random tree generated by a mechanism which picks a random leaf to extend on. Our {\it multi-type Yule tree with mutations} is a continuous time model and is a multi-type generalization of the tree generated by a pure birth process. In order to investigate their topological features we analyze the number of different types of cherries and different types of pendants in the tree: {\it cherries} are pairs of leaves that are only one edge away from each other, and {\it pendants} are leaves that are more than a single edge away from another leaf. We use the random recursive mechanisms for generating splits in the trees to obtain exact results for finite sized trees, as well as asymptotic results as the trees grow in size. The distribution of the number of pendants and cherries in the tree reflects the model parameters and can be used to infer them. 

For the multi-type ERM random trees, we identify the mean number of cherries and their variance (Propositions~\ref{prop:meanscherries} and ~\ref{prop:varcherries}), and also derive asymptotic results as the number of leaves in the tree grows (Theorems~\ref{thm:polya-coro-1} and ~\ref{thm:polya-coro-2}). We use the limiting fraction of different cherries to infer the multi-type probabilities in the model (Corollary~\ref{cor:ERMinferrence}). Examples of particular models for multi-type ERM trees are discussed in Section~\ref{subsec:ERMcases}.
For the multi-type birth-and-death process we first identify the process obtained by pruning away the extinct lineages (Proposition~\ref{prop:ancestral}) as a specific type of a multi-type Yule tree with mutations. Using the distribution of types at the leaves (Lemma ~\ref{lem:meannudependant-eq}) we find the distribution of different cherries and pendants (Propositions~\ref{prop:meanmudependant} and ~\ref{prop:meangamma}) in a general multi-type Yule tree with mutations. We also derive their asymptotics in the long term limit (Theorem~\ref{thm:limit-eta-tdepnotcomm}) and provide the way in which the original speciation and extinction rates can be inferred from these topological features (Corollary~\ref{cor:contsinferrence}). 

\section{Multi-type ERM trees}

Consider a (single type) random tree constructed recursively, from a single node leaf, by picking at each step one leaf uniformly at random and creating a branch-point by attaching two new leaves to it. The distribution of this tree is called `equal rates Markov' (ERM) (first investigated in \cite{Harding}) and has a long list of mathematical results associated to it (\cite{Aldous1, Aldous2}). Trees with this distribution can be generated in a number of different ways, forwards in time - by using a (pure birth) Yule process stopped the first time it reaches a prescribed number of leaves and ignoring the random lengths of its branches, or backwards in time - starting from a prescribed number of leaves using a neutral (coalescent) Moran process. They have been used in numerous studies as a null model in investigating pattterns in tree shapes (\cite{MooersHeard}). In terms of its statistical features the number of cherries $C_n$ for a tree with $n$ leaves is known (\cite{McKSteel}) to have the following properties:
$$\mbb{E}[C_n]=\frac{n}{3}\,\,\mbox{ for }n\geq 3,\,\,\,\,\,\mbb{V}[C_n]=\frac{2n}{45}\,\,\mbox{ for }n\geq 5;$$
and their distribution satisfies a central limit theorem:
$$\frac{C_n-n/3}{\sqrt{2n/45}}\Rightarrow N(0,1).$$
Their results were shown using an extended P\'olya urn process (see \cite{S96} or \cite{J04}).

We consider a multi-type version of this random tree, where each node (branch-points and leaves) has a type $k\in\mathcal{K}$ associated with it. The shape of the tree is constructed in the same way as in the single type process, with each leaf having the same chance, regardless of its type, of being picked at random to create the next branch-point with two new leaves attached to it. The types of the two leaves being attached, however, depend on the type of the leaf that they are being attached to. For each type $i,j_1,j_2\in\mathcal{K}$ the probabilities $q^{j_1,j_2}_i$ determine the chance that a leaf of type $i$ has types $j_1,j_2$ attached to it. Since we do not distinguish between different embeddings of the tree in the plane, we can w.l.o.g. assume $j_1\le j_2$. We call this tree a {\it multi-type ERM tree}. The random tree with types is distributed as a Markov field (with propagation matrix $\{q^{j_1, j_2}_i\}_{i,j_1\le j_2\in\mathcal{K}}$) on an ERM tree. Note that, for each $i$, $\sum_{j_1\le j_2}q^{j_1,j_2}_i=1$. In order to avoid trivial cases that generate only single type trees we will assume throughout that $q^{ii}_i\neq 1, \forall i$.

\begin{figure}
\centering
{\includegraphics[width=10cm, trim =0cm 0.5cm 0cm 0cm, clip=true]{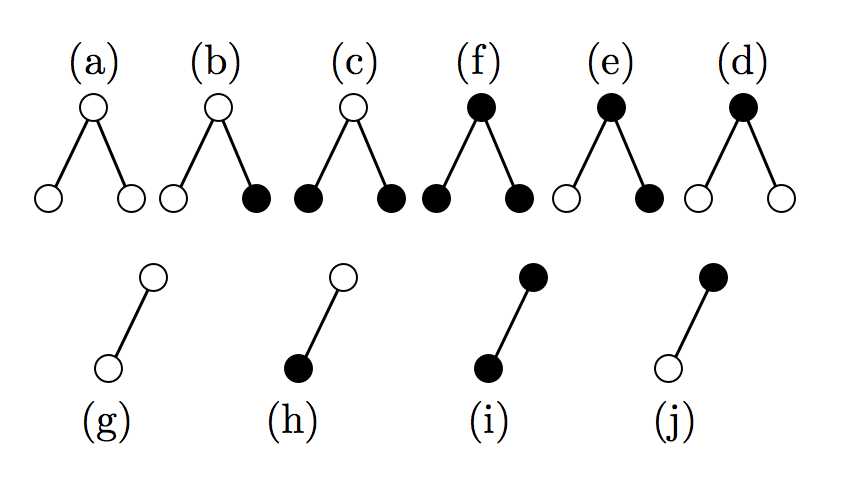}}\\
\caption{Type $1$ is denoted by a blank circle, and type $2$ by a full circle; different types of cherries: (a) type 111, (b) type 112, (c) type 122, (d) type 211, (e) type 212 and (f) type 222; and different types of pendants: (g) type 11, (h) type 12, (i) type 21 and (j) type 22.}
\label{fig:cherriestypes}
\end{figure}
\begin{figure}
\centering
{\includegraphics[width=5cm, trim =0cm 0.5cm 0cm 0cm, clip=true]{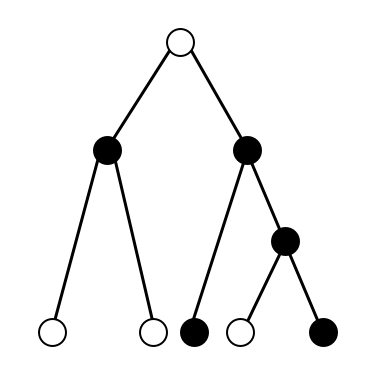}}\\
\caption{Tree with $N_1(5)=3,N_2(5)=2$, one cherry of type $211$, one cherry of  type $212$, and one pendant of type $22$.}
\label{fig:treecherries}
\end{figure}
For the sake of simplicity we consider $\mathcal{K}=\{1,2\}$. There are $k^2(k+1)/2=6$ different types of cherries $\{111,112,122, 211, 212, 222\}$ and $k^2=4$ different types of pendants $\{11, 12, 21, 22\}$, as illustrated in Figure~\ref{fig:cherriestypes}. Figure~\ref{fig:treecherries} illustrates cherries in an example of a tree with $n=5$ leaves.

\subsection{Moments of the number of different types of cherries}\label{subsec:ERMmoments}
For a tree with $n$ leaves we let $N_1(n)$ denote the number of leaves of type 1, $N_2(n)=n-N_1(n)$ the number of leaves of type 2, and  $C_{i}^{j_1j_2}(n)$  the number of cherries of type $ij_1j_2$. 
Their means are relatively straightforward to calculate.

\begin{lem}\label{lem:meanleaves}
Assume the probabilities $\{q_i^{j_1j_2}\}_{i,j_1\le j_2\in\{1,2\}}$ satisfy \mbox{\rm{($\star$)}: $c_1-c_2\notin\{-2,-2\}$} for  $c_1:=2q_1^{11}+q_1^{12}$ and $c_2:=2q_2^{11}+q_2^{12}$. Then, $\forall n\ge3$,
$$\nu_1(n):=\mbb{E}[N_1(n)]=\frac{c_2n}{2-c_1+c_2}-\frac{(2c_2 -(2-c_1+c_2)\nu_1(2))\Gamma(n-1+c_1-c_2)}{(2-c_1+c_2)\Gamma(c_1-c_2+2)\Gamma(n)},$$
where $\Gamma(n)$ is the gamma function, and $\nu_1(2)= \left\{ \begin{array}{cc}\!\!c_1, & \mbox{if }N_1(1)=1 \mbox{ (initial leaf type is 1)}\\ \!\!c_2, &\mbox{if }N_2(1)=1 \mbox{ (initial leaf type is 2)}\end{array}\right.$\!\!.\\
Analogous formula holds for $\nu_2(n):=\mbb{E}[N_2(n)]$ in which:\, $c_1$ is replaced by $c_1':=2q_2^{22}+q_2^{12}$ $(=2-c_2)$, $c_2$ is replaced by $c_2':= 2q_1^{22}+q_1^{12}$ $(=2-c_1)$ (notice $c_1'-c_2'=c_1-c_2$ \mbox{remains the same}), and $\nu_1(2)$ is replaced by $\nu_2(2)=\left\{ \begin{array}{cc}\!\!c_1', & \mbox{if }N_2(1)=1 \mbox{ (initial leaf type is 2)}\\ \!\!c_2', &\mbox{if }N_1(1)=1 \mbox{ (initial leaf type is 1)}\end{array}\right.$\!\!.\\
\end{lem}

\vspace{-5mm}
\begin{proof}
The result follows from a straightforward recursion, for any $2\le n_1\le n$, we have
\begin{eqnarray*}
\mbb{P}[N_1(n)=n_1]&=&\left(\frac{n_1q_1^{12}}{n-1}+\frac{(n-n_1-1)q_2^{22}}{n-1}\right)\mbb P[N_1(n-1)=n_1]\\
&+&\left(\frac{(n_1-1)q_1^{11}}{n-1}+\frac{(n-n_1)q_2^{12}}{n-1}\right)\mbb P[N_1(n-1)=n_1-1]\\
&+&\left(\frac{(n-n_1+1)q_2^{11}}{n-1}\right)\mbb P[N_1(n-1)=n_1-2]\\
&+&\left(\frac{(n_1+1)q_1^{22}}{n-1}\right)\mbb P[N_1(n-1)=n_1+1].\\
\end{eqnarray*}
This yields a recurrence relation for the generating function $G_n(x)=\sum_{n_1\geq0}\mbb{P}[N_1(n)=n_1]x^{n_1}$, which when differentiated and evaluated at $x=1$ results in the recurrence relation for $\nu_1(n)$
$$\nu_1(n+1)=(q_2^{12}+2q_2^{11})+\frac1n(n+q_1^{12}+2q_1^{11}-q_2^{12}-2q_2^{11}-1)\nu_1(n)$$
and solving it we obtain the claimed result.
\end{proof}

\begin{rem}\label{rem:meansspecial}
The condition $c_1-c_2\neq 2$ rules out trivial cases generating single type trees $\{q_1^{11}=1,q_2^{22}=1\}$ of only type 1 or type 2 (depending on initial type). The condition $c_1-c_2\neq -2$  rules out the unusual special case of completely alternating types $\{q_1^{22}=1, q_2^{11}=1\}$. 
However, a number of interesting cases are covered by our results, as shown at the end of this Section.
\end{rem}

\begin{prop}\label{prop:meanscherries}
Under the same conditions ($\star$) as in Lemma~\ref{lem:meanleaves}, $\forall n\geq3$, for 
$$\mu_{1}^{11}(n):=\mbb{E}[C_{1}^{11}(n)],\;\; \mu_{1}^{12}(n):=\mbb{E}[C_{1}^{12}(n)],\;\; \mu_{1}^{22}(n):=\mbb{E}[C_{1}^{22}(n)]$$
we have
\begin{eqnarray*}
\mu_{1}^{11}(n)&=&\frac{3(2-c_1+c_2)(2\mu_{1}^{11}(3)-q_1^{11}\nu_1(2))+n(n-1)(n-2)q_1^{11}c_2}{3(2-c_1+c_2)(n-1)(n-2)}-q_1^{11} C(n)\\
\mu_{1}^{12}(n)&=&\frac{3(2-c_1+c_2)(2\mu_{1}^{12}(3)-q_1^{12}\nu_1(2))+n(n-1)(n-2)q_1^{12}c_2}{3(2-c_1+c_2)(n-1)(n-2)}-q_1^{12}C(n)\\
\mu_{1}^{22}(n)&=&\frac{3(2-c_1+c_2)(2\mu_{1}^{22}(3)-q_1^{22}\nu_1(2))+n(n-1)(n-2)q_1^{22}c_2}{3(2-c_1+c_2)(n-1)(n-2)}-q_1^{22}C(n)
\end{eqnarray*}
where $\nu_1(2),c_1,c_2$ are as in Lemma~\ref{lem:meanleaves}, the constants $C(n)$ are 
$$C(n):=\frac{(2c_2-(2-c_1+c_2)\nu_1(2))\Gamma(n-1+c_1-c_2)}{(2-c_1+c_2)\Gamma(c_1-c_2+2)\Gamma(n)},$$ 
and the initial values are
$$\mu_{1}^{11}(3)= \left\{ \begin{array}{cc}(q_{1}^{11})^2+q_1^{11}q_1^{12}/2 & \mbox{ if }N_1(1)=1\\ q_{1}^{11}q_2^{11}+q_1^{11}q_2^{12}/2 &\mbox{ if }N_2(1)=1\end{array},\right.\,\mu_{1}^{12}(3)= \left\{ \begin{array}{cc}(q_1^{12})^2/2+q_1^{11}q_1^{12} & \mbox{ if }N_1(1)=1\\ q_2^{11}q_1^{12}+q_2^{12}q_1^{12}/2 &\mbox{ if }N_2(1)=1\end{array}\right.$$ 
and
$$\mu_{1}^{22}(3)= \left\{ \begin{array}{cc}q_1^{11}q_1^{22}+q_1^{12}q_1^{22}/2 & \mbox{ if }N_1(1)=1\\ q_2^{12}q_1^{22}/2+q_2^{11}q_1^{22} &\mbox{ if }N_2(1)=1\end{array}.\right.$$
Analogous formulae hold for $$\mu_{2}^{11}(n):=\mbb{E}[C_{2}^{11}(n)],\;\; \mu_{2}^{12}(n):=\mbb{E}[C_{2}^{12}(n)],\;\; \mu_{2}^{22}(n):=\mbb{E}[C_{2}^{22}(n)]$$ in which:\, probabilities $q^{j_1j_2}_1$ are replaced by $q^{j_1j_2}_2$, $\nu_1(2)$ is replaced by $\nu_2(2)=2-\nu_1(2)$,  $c_1$ and $c_2$ are replaced by $c_1'$ and $c_2'$ respectively (as in Lemma~\ref{lem:meanleaves}), the constants $C(n)$ remain the same if the initial type is interchanged, and $\mu_1^{j_1j_2}(3)$ are replaced by $\mu_2^{j_1j_2}(3)$ obtained by fully interchanging types in the formulae for $\mu_1^{j_1j_2}(3)$.
\end{prop}

\begin{proof}
Since at each step new leaves are attached in pairs, there is no need to keep track of the number of different pendants. It suffices to keep track of the number of different types of leaves and only of the cherries of the specific type we are trying to calculate.

Let $$f_n^{ij_1j_2}(n_1,k):=\mbb P[N_1(n)=n_1,C_i^{j_1j_2}(n)=k]$$ denote the joint probability function for $N_1(n), C_i^{j_1j_2}(n)$ and $$F^{ij_1j_2}_n(x,y)=\sum_{n_1\geq0,k\geq0}\mbb P[N_1(n)=n_1,K_i^{j_1j_2}(n)=k]x^{n_1}y^{k}$$ its generating function.
Using recursion arguments, for any $3\le n_1\le n, k\ge 1$, we have
\begin{eqnarray*}
f_n^{111}(n_1,k)\!\!\!&=&\!\!\!\Big(\frac{2kq_1^{11}}{n-1}+\frac{(n-n_1)q_2^{12}}{n-1}\Big)f_{n-1}^{111}(n_1-1,k)
+\Big(\frac{(n_1-2k)q_1^{12}}{n-1}+\frac{(n-n_1-1)q_2^{22}}{n-1}\Big)f_{n-1}^{111}(n_1,k)\\
&+&\!\!\!\Big(\frac{(n-n_1+1)q_2^{11}}{n-1}\Big)f_{n-1}^{111}(n_1-2,k)
+\Big(\frac{(n_1+1-2k)q_1^{22}}{n-1}\Big)f_{n-1}^{111}(n_1+1,k)\\
&+&\!\!\!\Big(\frac{(n_1-1-2(k-1))q_1^{11}}{n-1}\Big)f_{n-1}^{111}(n_1-1,k-1)
+\Big(\frac{2(k+1)q_1^{22}}{n-1}\Big)f_{n-1}^{111}(n_1+1,k+1)\\
&+&\!\!\!\Big(\frac{2(k+1)q_1^{12}}{n-1}\Big)f_{n-1}^{111}(n_1,k+1),\\
\end{eqnarray*}
for cherries of type $111$, 
\begin{eqnarray*}
f_n^{112}(n_1,k)\!\!\!&=&\!\!\!\Big(\frac{kq_1^{12}}{n-1}+\frac{(n-n_1-k-1)q_2^{22}}{n-1}\Big)f_{n-1}^{112}(n_1,k)\\
&+&\!\!\!\Big(\frac{(n_1-k-1)q_1^{11}}{n-1}+\frac{(n-n_1-k)q_2^{12}}{n-1}\Big)f_{n-1}^{112}(n_1-1,k)
+\Big(\frac{(n-n_1-k+1)q_2^{11}}{n-1}\Big)f_{n-1}^{112}(n_1-2,k)\\
&+&\!\!\!\Big(\frac{(k+1)q_2^{11}}{n-1}\Big)f_{n-1}^{112}(n_1-2,k+1)
+\Big(\frac{(k+1)q_2^{12}}{n-1}\\
&+&\!\!\!\frac{(k+1)q_1^{11}}{n-1}\Big)f_{n-1}^{112}(n_1-1,k+1)
+\Big(\frac{(k+1)q_2^{22}}{n-1}\Big)f_{n-1}^{112}(n_1,k+1)\\
&+&\!\!\!\Big(\frac{(n_1-k+1)q_1^{12}}{n-1}\Big)f_{n-1}^{112}(n_1,k-1)
+\Big(\frac{(n_1-k+1)q_1^{22}}{n-1}\Big)f_{n-1}^{112}(n_1+1,k)\\
&+&\!\!\!\Big(\frac{(k+1)q_1^{22}}{n-1}\Big)f_{n-1}^{112}(n_1+1,k+1)
\end{eqnarray*}
for cherries of type $112$, and 
\begin{eqnarray*}
f_n^{122}(n_1,k)\!\!\!&=&\!\!\!\Big(\frac{(n-n_1-2k)q_2^{12}}{n-1}+\frac{(n_1-1)q_1^{11}}{n-1}\Big) f_{n-1}^{122}(n_1-1,k)\\
&+&\!\!\!\Big(\frac{(n-n_1-2k-1)q_2^{22}}{n-1}+\frac{n_1q_1^{12}}{n-1}\Big)f_{n-1}^{122}(n_1,k)
+\Big(\frac{2(k+1)q_2^{11}}{n-1}\Big)f_{n-1}^{122}(n_1-2,k+1)\\
&+&\!\!\!\Big(\frac{2(k+1)q_2^{12}}{n-1}\Big)f_{n-1}^{122}(n_1-1,k+1)
+\Big(\frac{2(k+1)q_2^{22}}{n-1}\Big)f_{n-1}^{122}(n_1,k+1)\\
&+&\!\!\!\Big(\frac{(n_1+1)q_1^{22}}{n-1}\Big) f_{n-1}^{122}(n_1+1,k-1)
+\Big(\frac{(n-n_1-2k+1)q_2^{11}}{n-1}\Big)f_{n-1}^{122}(n_1-2,k).
\end{eqnarray*}
for cherries of type $122$. Each of these equations yields a recurrence relation for the corresponding joint generating function $F^{ij_1j_2}_n(x,y)$ by summing over $n_1$ and $k$. Differentiating and evaluating them at \mbox{$x=y=1$} then provides recurrences for the means $\mu_{i}^{j_1,j_2}(n)$ 
\begin{eqnarray*}
\mu_1^{11}(n+1)&=&\frac{n-2}{n}\mu_1^{11}(n)+\frac{q_1^{11}}{n}\nu_1(n)\\
\mu_1^{12}(n+1)&=&\frac{n-2}{n}\mu_1^{12}(n)+\frac{q_1^{12}}{n}\nu_1(n)\\
\mu_1^{22}(n+1)&=&\frac{n-2}{n}\mu_1^{22}(n)+\frac{q_1^{22}}{n}\nu_1(n)
\end{eqnarray*}
solving which, with the expression for $\nu_1(n)$ from Lemma~\ref{lem:meanleaves}, gives the claimed formulae.
\end{proof}

\begin{rem}
Simple algebra shows that the mean numbers of all cherries $\sum_{i,j_1\le j_2}\mu_1^{j_1j_2}(n)$ add up to $n/3$, corresponding to the known mean  number of cherries in a single-type ERM tree. 
\end{rem}

\begin{prop}
\label{prop:varcherries}
Assume that $c_1-c_2\notin\{-2,-1,0,1,3/2,2\}$ for $c_1,c_2$ as in Lemma~\ref{lem:meanleaves}. Then, $\forall n\geq5$, for 
$$\sigma_{1}^{11}(n):=\mbb{V}[C^{11}_1(n)], \;\; \sigma_{1}^{12}(n):=\mbb{V}[C^{12}_1(n)], \;\; \sigma_{1}^{22}(n):=\mbb{V}[C^{22}_1(n)]$$
we have
$$\sigma_{1}^{11}(n), \sigma_{1}^{12}(n), \sigma_{1}^{22}(n)\sim\mathrm{O}(n)+\mathrm{O}(n^{c_1-c_2-1})+\mathrm{O}(n^{2(c_1-c_2-1)}).$$
The same asymptotics hold for $$\sigma_{2}^{11}(n):=\mbb{V}[C_{2}^{11}(n)],\;\; \sigma_{2}^{12}(n):=\mbb{V}[C_{2}^{12}(n)],\;\; \sigma_{2}^{22}(n):=\mbb{V}[C_{2}^{22}(n)]$$  as the exponents $c_1'-c_2'=c_1-c_2$ are the same in these cases.
\end{prop}

\begin{proof}
Using recurrence relations for the generating functions $F^{111}_n(x,y),\,F^{112}_n(x,y), F^{122}_n(x,y)$ from the proof of Proposition~\ref{prop:meanscherries}, taking second derivatives in $x,y$ and evaluating them at $x=y=1$ yields recurrence equations for the variances for the number of cherries for each of the types $111,112$ and $113$, respectively. For $j_1\le j_2\in\{1,2\}$ let 
$$R^{1j_1j_2}_{yy}(n):=\frac{\partial^2F^{1j_1j_2}_n(x,y)}{\partial y^2}\Big|_{x=1,y=1}, \quad R^{1j_1j_2}_{xy}(n):=\frac{\partial^2F^{1j_1j_2}_n(x,y)}{\partial x\partial y}\Big|_{x=1,y=1}.$$ 
From equations for $F^{1j_1j_2}_n(x,y)$ we obtain the recurrence relation for $R^{1j_1j_2}_{yy}(n), R^{1j_1j_2}_{xy}(n)$ as
$$
R^{1j_1j_2}_{yy}(n)=\frac{\Gamma(n-4)}{\Gamma(n)}\Big(\sum_{n_1=1}^{n-1}\frac{2q_1^{j_1j_2}(R^{1j_1j_1}_{xy}(n_1)-2\mu_1^{j_1j_2}(n_1))\Gamma(n_1+1)+24R^{1j_1j_2}_{yy}(5) }{n_1\Gamma(n_1-3)} \Big)
$$
The equations for $R^{1j_1j_2}_{xy}(n)$ satisfy the recurrence relations, for any $n\ge 4$
\begin{eqnarray*}
R^{111}_{xy}(n+1)\!\!\!\!\!&=&\!\!\!\!\!\frac1n\Big((2-2c_1+c_2n)\mu_1^{11}(n)+(n-3+c_1-c_2)R^{111}_{xy}(n)+2q_1^{11}\nu_1(n)+q_1^{11}R(n)\Big)\\
R^{112}_{xy}(n+1)\!\!\!\!\!&=&\!\!\!\!\!\frac1n\Big((1-c_1+c_2n-c_2)\mu_1^{12}(n)+(n-3+c_1-c_2)R^{112}_{xy}(n)+2q_1^{12}\nu_1(n)+q_1^{11}R(n)\Big)\\
R^{122}_{xy}(n+1)\!\!\!\!\!&=&\!\!\!\!\!\frac1n\Big((c_2n-2c_2)\mu_1^{22}(n)+(n-3+c_1-c_2)R^{112}_{xy}(n)+q_1^{22}R(n)\Big)
\end{eqnarray*}
where $R(n)=\frac{\partial^2 G_n}{\partial x^2}|_{x=1,y=1}$ is the second moment for the number of leaves $N_1$ of type $1$ (see Lemma~\ref{lem:meanleaves}) and satisfies the recurrence, for $n\ge 3$
$$
R(n)=\frac1n\Big(2nq_2^{11}+(2nc_2+2q_1^{11}-6q_2^{11}-2q_2^{12})\nu_1(n)+(n-2+2c_1-2c_2)R(n)\Big)
$$
From this we have the variances to be
$$
\sigma_i^{j_1j_2}(n)=R^{ij_1j_2}_{yy}(n)+\mu_i^{j_1j_2}(n)-(\mu_i^{j_2j_2}(n))^2.
$$
Explicitly solving these equations for the variances requires much more complicated calculations than for the means. Using Maple yields formulae which are quite long and cluttered. However,  expanding these formulae with respect to $n$ we obtain the asymptotic results above.
\end{proof}

\subsection{Special cases of multi-type ERM models}\label{subsec:ERMcases}
\vspace{-2mm}
To illustrate how multi-type ERM trees for different speciation models give different cherry distributions, we 
consider particular cases corresponding to specific values of $\{q_i^{j_1j_2}\}_{i,j_1\le j_2\in\{1,2\}}$:\\
\indent (a) The {\it `single type'} model is the trivial one in which the only type in the tree is the initial one:\, $q_1^{11}=q_2^{22}=1$ \;($c_1-c_2=2$);\\
\indent (b) The `{\it alternating type'} model is one in which one type can only attach to itself leaves of the other type:\, $q_1^{22}=q_2^{11}=1$ \;($c_1-c_2=-2$)\\
\indent (c) In the {\it `neutral to type'} model the branch-point type does not determine the probabilities of leaf types:\, for each $j_1\le j_2$:\, $q_i^{j_1j_2}$ is independent of whether $i=1,2$\; ($c_1-c_2=0, c_1+c_1'=2$);\\
\indent (d) In the {\it `only mixed type'} model each type has only mixed types attached to it:\, $q_i^{12}=1$ for both $i=1,2$\; ($c_1-c_2=0$);\\
\indent (e) The {\it `asymmetric change in type'} represents a model where one type can be randomly gained from the other but once gained can no longer be lost:\, $\{q_1^{11}=1, q_2^{11}=q_2^{22}=(1-q_2^{12})/2\}$ or $\{q_2^{22}=1, q_1^{11}=q_1^{22}=(1-q_1^{12})/2\}$\; ($c_1-c_2=1$);


\label{table:mean-particular}
\begin{tabular}{|l|c|c|} 
\hline & \\
{\it Single type}: & $\displaystyle\mu_1^{11}(n)=\frac{n}{3}$\, all other $\displaystyle\mu_i^{j_1j_2}=0$,\, if $N_1(1)=1$ \\  \quad\qquad\qquad $q_1^{11}=q_2^{22}=1$ ($ c_1-c_2=2$) &
\\   &	$\displaystyle\mu_2^{22}(n)=\frac{n}{3}$\, all other $\displaystyle\mu_i^{j_1j_2}=0$,\, if $N_2(1)=1$  \\
						 &\\ \hline
						 &\\	
{\it Alternating type}:  & 
\\  \quad\qquad\quad $q_1^{22}=q_2^{11}=1$ ($ c_1-c_2=-2$) & $\displaystyle\mu_1^{22}(n)=\mu_2^{11}(n)=\frac{n}{6}$,\, all other $\displaystyle\mu_i^{j_1j_2}=0$ 	 \\
						 &\\ \hline
						 &\\	
{\it Neutral to type}:    & $\mu_1^{11}(n)=\displaystyle\frac{nq_1^{11}c_1}{6}$,\;\; $\mu_2^{11}(n)=\displaystyle\frac{nq_2^{11}c_1'}{6}$, \\ &
\\ \quad\qquad $q_1^{11}=q_2^{11},\, q_1^{12}=q_2^{12},\, q_1^{22}=q_2^{22}$   & $\mu_1^{12}(n)=\displaystyle\frac{nq_1^{12}c_1}{6}$,\;\; $\mu_2^{12}(n)=\displaystyle\frac{nq_2^{12}c_1'}{6}$,\\ & \\ \,\;\qquad\qquad\qquad\qquad\qquad ($ c_1-c_2=0$) & $\mu_1^{22}(n)=\displaystyle\frac{nq_1^{22}c_1}{6}$,\;\; $\mu_2^{22}(n)=\displaystyle\frac{nq_2^{22}c_1'}{6}$\\
				 		 &   \\ \hline  
						        &\\					       
{\it Only mixed type}:  &   
\\ \;\;\qquad\qquad $q_1^{12}=q_2^{12}=1$ ($ c_1-c_2=0$) &  $\displaystyle\mu_1^{12}(n)=\mu_2^{12}(n)=\frac{n}{6}$,\, all other $\displaystyle\mu_i^{j_1j_2}=0$\\ 
    						 &   \\ \hline  
						        &\\
{\it Asymmetric change}:  & $\displaystyle\mu_1^{11}(n)=\frac{n}{3}$,\, all other $\displaystyle\mu_i^{j_1j_2}=0$,\,\\ &  \qquad\qquad\qquad\qquad\qquad\qquad\;\; if $N_1(1)=1$\\ 
\quad\qquad $q_1^{11}=1, q_2^{11}=q_2^{22}$ ($ c_1-c_2=1$) & 	 $\displaystyle\mu_1^{11}(n)=\frac{n}{3}-\frac12$,\, $\displaystyle\mu_2^{11}=\mu_2^{22}=\frac14(1-q_2^{12})$,\\ 
& \qquad\qquad\qquad\qquad$\displaystyle\mu_2^{12}=q_2^{12}$,\, if $N_2(1)=1$\\
						    &\\ \hline
\end{tabular}\\

For cases (a),(b) we could not use Proposition~\ref{prop:meanscherries} and we calculated the means directly.
The value of $c_1-c_2=c_1'-c_2'$ reflects the tendency of leaves to attach to leaves of their own type - the higher it is, the more weight is given to attaching to leaves of its own type as opposed to leaves of the opposite type (the two extreme cases are the single type and the alternating type). The sum of means for all different types of cherries coincides with the mean ($n/3$) of a (single-type) ERM tree as found by McKenzie and Steel \cite{McKSteel}.
For these cases we can get exact values for the variances of the numbers of cherries (directly from their generating functions) instead of relying only on asymptotics as in Proposition~\ref{prop:varcherries}. Note that the sum of variances for all different types of cherries only coincides with the variance ($2n/45$) of a single-type ERM tree \cite{McKSteel} in the extreme cases (a),(b) when the covariances are zero. 
\label{table:variance-particular}
\begin{tabular}{|l|c|c|} 
\hline & \\
{\it Single type}:  & $\displaystyle\sigma_1^{11}(n)=\frac{2n}{45}$,\, all other $\displaystyle\sigma_i^{j_1j_2}(n)=0$, if $N_1(1)=1$ \\ \quad\qquad\qquad $q_1^{11}=q_2^{22}=1$ ($ c_1-c_2=2$)   &
\\    &	$\displaystyle\sigma_2^{22}(n)=\frac{2n}{45}$,\, all other $\displaystyle\sigma_i^{j_1j_2}(n)=0$, if $N_2(1)=1$ \\
						 &\\ \hline
						 &\\
{\it Alternating type}:  & 
\\  \quad\qquad\quad $q_1^{22}=q_2^{11}=1$ ($ c_1-c_2=-2$) & $\displaystyle\sigma_1^{22}(n)=\sigma_2^{11}(n)=\frac{2n}{90}$,\, all other $\displaystyle\sigma_i^{j_1j_2}=0$ 	 \\
						 &\\ \hline
						 &\\	
{\it Neutral to type}:   & $\sigma_1^{11}(n)=\displaystyle\frac{nq_1^{11}(6(q_1^{11})^2+15c_1-8q_1^{11}c_1^2)}{90}$,\\ & 
\\ \quad\qquad $q_1^{11}=q_2^{11},\, q_1^{12}=q_2^{12},\, q_1^{22}=q_2^{22}$  & 	$\sigma_1^{12}(n)=\displaystyle\frac{nq_1^{12}(6q_1^{11}q_1^{12}+15c_1-8q_1^{12}c_1^2)}{90}$,\\ & \\
\,\;\qquad\qquad\qquad\qquad\qquad ($ c_1-c_2=0$)  & 	$\sigma_1^{22}(n)=\displaystyle\frac{nq_1^{22}(6q_1^{11}q_1^{22}+15c_1-8q_1^{22}c_1^2)}{90}$,\\
						       &   \\ \hline  
						       &\\
{\it Only mixed type}:  &   
\\ \;\;\qquad\qquad $q_1^{12}=q_2^{12}=1$ ($ c_1-c_2=0$) &	$\displaystyle\sigma_1^{12}(n)=\sigma_2^{12}(n)=\frac{7n}{90}$, \, all other $\displaystyle\sigma_i^{j_1j_2}(n)=0$ \\
    						       &   \\ \hline  
						       &\\				
{\it Asymmetric change} & $\displaystyle\sigma_1^{11}(n)=\frac{2n}{45}$,\, all other $\displaystyle\sigma_i^{j_1j_2}=0$,\\ &  \qquad\qquad\qquad\qquad\qquad\qquad\qquad\qquad if $N_1(1)=1$\\ & \\
\quad\qquad $q_1^{11}=1, q_2^{11}=q_2^{22}$ ($ c_1-c_2=1$) & $\displaystyle\sigma_1^{11}(n)=\frac{2n}{45}+{\rm o(n)}$, \; $\sigma_2^{11}(n)=\sigma_2^{22}(n)=\frac{1}{16}(1-q_2^{11})^3\!\!,$\\ & \\ & \qquad\qquad $\sigma_2^{12}(n)=\frac14(q_2^{11})^2(1-q_2^{11})$,\, if $N_2(1)=1$\\
						 &\\ \hline
\end{tabular}\\

\subsection{Asymptotic results for the number of cherries and pendants}\label{subsec:ERMlimits}

To consider the full structure (with correlations) of all the cherries in a multi-type ERM, we also need to keep track of the number of different pendants $L_i^j(n)$ of type $ij$ in a tree with $n$ leaves.  Let $\bm X(n)$ be a single vector representing different types of cherries and pendants\\
\begin{equation*}\label{eq:Polyavector}
\bm X(n)=(C_1^{11}(n), C_1^{12}(n), C_1^{22}(n), C_2^{22}(n), C_2^{12}(n), C_2^{11}(n), L_1^1(n), L_1^2(n), L_2^2(n), L_2^1(n))
\end{equation*}
Its asymptotic behaviour as $n\to\infty$ can be characterized in terms of a strong law. 
\begin{thm}
\label{thm:polya-coro-1}
Assume the probabilities $\{q_i^{j_1j_2}\}_{i,j_1\le j_2\in\{1,2\}}$ are such that to every cherry it is possible to eventually attach every other cherry \mbox{\rm{($\ast$)}}. Then, as ${n\rightarrow\infty}$

$$\frac{\bm X_n}{n}\mathop{\longrightarrow}\limits^{ \mbox{ a.s }} \bm v_1:=\frac{1}{3(2-c_1+c_2)}\left[\displaystyle\begin{array}{c}q_1^{11}c_2\\
						    q_1^{12}c_2\\
						   q_1^{22}c_2\\
						   q_2^{11}(2-c_1)\\
						  q_2^{12}(2-c_1)\\
						  q_2^{22}(2-c_1)\\
						  (c_1c_2)/2\\
						   (2-c_1)c_2/2\\ 
						   (2-c_1)(2-c_2)/2\\ 
						      (2-c_1)c_2/2 \end{array}\right] .$$
where $c_1:=2q_1^{11}+q_1^{12}, c_2:=2q_2^{11}+q_2^{12}$ are as in Lemma~\ref{lem:meanleaves}. 
\end{thm}

\begin{rem}
The condition \mbox{\rm{($\ast$)}} is a form of {\it irreducibility} of the cherry state space. It can be relaxed for multi-type tree models in which certain types of cherries are not at all appearing in the tree. We get the same strong law results on a state space (the vector $\bm X$, the matrix $\bm A$) that is restricted to the set of cherries that can appear in the tree, on which the condition \mbox{\rm{($\ast$)}} holds. 
\end{rem}

\begin{proof}
The proof relies on a P\'oya urn representation of the different types of cherries and pendants: 
an \emph{extended P\'olya urn process} $(\bm X(n))_{n\geq0}$ is a Markov chain on $\mbb Z_+^d$ where the coordinates of the random vector $\bm X(n)=(X_1(n),\ldots,X_d(n))$ represent the number of balls of type $i\in\{1,\ldots,d\}$ in an urn at step $n$. The process starts at $\bm X(0)$ and at each step balls of different types are added or removed from it. Each ball type has associated to it a positive weight $a_i\ge 0, i\in\{1,\ldots,d\}$ and a random vector $\bm \xi_i=(\xi_{i1},\ldots, \xi_{il})$ taking values in  $\mbb Z_+^d$, such that:
$\xi_{ij}\geq0,\forall j\neq i\mbox{ and }\xi_{ii}\geq-1, \forall i$ as well as $\mathbb{E}(\xi_{ij}^2)<\infty$.

The weights and random vectors together characterize the distribution of the transition matrix for the Markov chain:\\
\indent (i) at each step a ball is randomly selected from the urn with the probability of selecting a ball of type $i$ proportional to its weight $a_i$, that is the probability of drawing a ball of type $i$ at time $n\geq1$ is ${a_iX_{(n-1)i}}/{\sum_j a_jX_{(n-1)j}}$;\\
\indent (ii) if a ball of type $i$ was selected, then the number of balls of different types to be added to the urn is drawn according to the distribution $\xi_{ij}, j=1,\ldots,d$.
The condition $\xi_{ii}\geq-1$ means the selected ball that is removed from the urn may or may not be replaced on that step. It is useful to assume the urn never becomes empty, $|\bm X(n)|>0, \forall n\ge0$. Let $\bm a=(a_1,\ldots, a_d)$. The generating matrix of a P\'olya urn is defined as $\bm A:=(a_j\E(\xi_{ji}))_{i,j=1}^d$, whose eigenvalues in decreasing order of real parts are denoted by $\lambda_1> \mathrm{Re}(\lambda_2)\ge \mathrm{Re}(\lambda_3) \cdots$ (Perron-Frobenius implies that $\lambda_1$ is real valued). The urn is called irreducible if, for any $i,j$, given the urn starts with a single ball of type $i$ it is eventually possible to add a ball of type $j$ to the urn.

An complete treatment of extended P\'olya urns is given in \cite{J04}. We state here only the results stated that are key for our proof. Assume the urn is such that: (a) it is irreducible; (b) $\lambda_1>0$, (c) $\lambda_1$ and $\lambda_2$ are simple eigenvalues with left and right eigenvectors $\bm u_1, \bm v_1$ and $\bm u_2, \bm v_2$ satisfying $\bm u_1 \cdot \bm v_1=\bm u_2 \cdot \bm v_2=1$ and $\bm a \cdot \bm v_1=1$; (d) $\mathrm{Re}(\lambda_2)> \mathrm{Re}(\lambda_3)$. The last condition implies that the set of eigenvectors $\lambda$ satisfying $\mathrm{Re}(\lambda)>\lambda_1/2$ consists either only of $\lambda_2$ or it is empty. Under these assumptions Theorem~ 3.21 of \cite{J04} insures that, in the limit as $n\to\infty$,
$$\frac{\bm X_n}{n}\mathop{\longrightarrow}\limits^{ \mbox{ a.s }} \bm v_1.$$

The process of constructing a multi-type ERM tree can be viewed as a P\'olya urn process: the balls of different types are all the different types of cherries and different types of pendants. For $\mathcal{K}=\{1,2\}$ we have $d= 10$, as shown in Figure 1. The ball types corresponding to any of the cherries have a weight $a_i=2$, and those corresponding to pendants have a weight $a_i=1$, as it is twice as likely to choose a cherry than a pendant when a leaf is picked uniformly at random. Careful consideration of the multi-type ERM construction rules (a cherry being selected means that a new pair of leaves is to be added to a randomly chose one of its leaves) shows that the generating matrix of this P\'olya urn process is:
$$\bm A:=\left[\begin{array}{ccccccccccc}
								-2(q_1^{12}+q_1^{22}) &  q_1^{11} &0  & 0 & q_1^{11} & 2 q_1^{11} &q_1^{11} &0 &0 & q_1^{11}  \\
                                   			2q_1^{12} & -(2-q_1^{12})& 0 & 0 &  q_1^{12} & 2q_1^{12} & q_1^{12} &0 &0 & q_1^{12}\\
				  			2q_1^{22} & q_1^{22} & -2 & 0      &  q_1^{22}& 2q_1^{22} & q_1^{22}&0&0 & q_1^{22} \\
							0 & q_2^{22} & 2q_2^{22} & -2(q_2^{11}+q_2^{12})& q_2^{22}& 0 & 0& q_2^{22}&q_2^{22}&0\\
							0 & q_2^{12} & 2q_2^{12} & 2q_2^{12}  & -(2-q_2^{12}) & 0 & 0& q_2^{12}&q_2^{12} &0\\
							0& q_2^{11} & 2q_2^{11} &2 q_2^{11} & q_2^{11} & -2 & 0 & q_2^{11}&q_2^{11}&0\\
							2 & 1 & 0 & 0 & 0 & 0 & -1 &0 & 0  &0\\
							0 & 1 & 2 & 0 & 0 & 0 & 0 &-1 & 0  &0\\
							0 & 0  &0 & 2 & 1 & 0  &0& 0&-1 &0\\
							0 & 0  &0 & 0 & 1 & 2  &0& 0&0 &-1\\\end{array}\right]$$
whose eigenvalues can be shown to be: $\lambda_1=1$, $\lambda_2=c_1-c_2-1=2q_1^{11}+q_1^{12}-2q_2^{11}-q_2^{12}-1$, $\lambda_3=\lambda_4=-1$ and $\lambda_5=\cdots=\lambda_{10}=-2$. 
The  normalized right and left eigenvectors of the largest real eigenvalue can be calculated in terms of the multi-type probabilities to be 
$$\bm v_1=\frac{1}{3(2-c_1+c_2)}\left[\displaystyle\begin{array}{c}
							   q_1^{11}c_2\\
						   q_1^{12}c_2\\
						   q_1^{22}c_2\\
						   q_2^{22}(2-c_1)\\
						   q_2^{12}(2-c_1)\\
						   q_2^{11}(2-c_1)\\
						   c_1c_2/2\\
						   (2-c_1)c_2/2\\ 
						   (2-c_1)(2-c_2)/2\\ 
						   (2-c_1)c_2/2 \end{array}\right],
						   \quad \bm u_1 = \left[\displaystyle\begin{array}{c} 2\\ 2\\ 2\\2\\ 2\\ 2\\ 1\\ 1\\ 1\\1\\ \end{array}\right]$$
If we assume that $q_1^{11}, q_2^{22}\neq1$, this excludes the case when the generated ERM tree is of a single type only, the urn process is irreducible, and also $\lambda_2<1$ is a simple eigenvalue. 
As all the assumptions are satisfied, applying the Theorem for P\'olya urns we obtain the claimed results.
\end{proof}

\begin{rem}
This agrees with our earlier result from Proposition~\ref{prop:meanscherries} as the means of the number of cherries obtained earlier in fact satisfy $\mbb{E}[\bm X_i(n)]/n\to \bm v_{1i}$, for $i=1,\ldots, 6$ as $n\to\infty$. Also, by restricting the state space to possible types of cherries the strong law result can be used on all of our special cases of multi-type ERM models, and compared to the means calculated for finite $n$, except for the case of `asymmetric change' in which type 111 is a sink for the process.
\end{rem}

A central limit law for its (normalized) asymptotic distribution holds as well.
\begin{thm}
\label{thm:polya-coro-2}
Assume $\{q_i^{j_1j_2}\}_{i,j_1\le j_2\in\{1,2\}}$ satisfy $q_1^{11}, q_2^{22}\neq1$, \mbox{and $c_1-c_2\neq 0$}. Then,\\ 
{\bf (i)} If $c_1-c_2=3/2$, as $n\rightarrow\infty$, 
$$\frac{\bm X_n-n\bm v_1}{n\ln(n)}\mathop{\Rightarrow}^{d}N(0,\bm \Sigma),$$
with
$$\bm \Sigma=C\left[\begin{array}{cccccccccc}(q_1^{11})^2&q_1^{11}q_1^{12}&q_1^{11}q_1^{22}&-q_1^{11}q_2^{22} &-q_1^{11}q_2^{12} &-q_1^{11}q_2^{11} &*&*&*&*\\												       \checkmark & (q_1^{12})^2& q_1^{12}q_1^{22}&-q_1^{12}q_2^{22} &-q_1^{12}q_2^{12} &-q_1^{12}q_2^{11} & * & *& * &*\\											       \checkmark&\checkmark& (q_1^{22})^2&-q_1^{22}q_2^{22} &-q_1^{22}q_2^{12} &-q_1^{12}q_2^{11}&*&*&*&*\\
												\checkmark&\checkmark&\checkmark& (q_2^{22})^2& q_2^{22}q_2^{12}& q_2^{22}q_2^{11}&*&*&*&*\\
												 \checkmark&\checkmark&\checkmark& \checkmark&(q_2^{12})^2& q_2^{12}q_2^{11}&*&*&*&*\\
												  \checkmark&\checkmark&\checkmark& \checkmark&\checkmark&(q_2^{11})^2&*&*&*&*\\												  					           *&*&*&*&*&*&*&*&*&*\\
												   *&*&*&*&*&*&*&*&*&*\\
												    *&*&*&*&*&*&*&*&*&*\\
												   *&*&*&*&*&*&*&*&*&*\\\end{array}\right]$$
where the constant $C$is given by 
$$C:=-8(9+12(q_1^{11})^2+2q_2^{11}q_1^{12}+4q_1^{11}q_2^{11}+4(q_1^{12})^2+14q_1^{11}q_1^{12}-4q_2^{11}-12q_1^{12}-21q_1^{11})/25;$$ the explicit expressions for entries marked by $*$ are omitted as they represent the covariances between cherries and the pendants; and expressions for the entries marked by a $\checkmark$ are omitted because they are follow from the symmetry of the covariance matrix.\\
{\bf (ii)} If $c_1-c_2<3/2$, as $n\rightarrow\infty$, 
$$\frac{\bm X_n-n\bm v_1}{\sqrt{n}}\mathop{\Rightarrow}^{d} N(0,\bm \Sigma').$$
where $\bm \Sigma'$ can be obtained explicitly only in some special cases.
\end{thm}

\begin{proof} The proof again relies on the corresponding result for our specific P\'olya urn described in the proof of Theorem~\ref{thm:polya-coro-1}: if we assume all the conditions there plus $\mathrm{Re}\lambda_2\le \lambda_1/2$, Theorems~3.22 and~3.23 of \cite{J04} insure that, as $n\to\infty$:\\
(i) if $\mathrm{Re}(\lambda_2)= \lambda_1/2$, then 
$$\frac{\bm X_n-n\lambda_1\bm v_1}{n\ln(n)}\mathop{\Rightarrow}^{d} N(0,\bm \Sigma),$$
where the covariance matrix is given by\;
$\bm \Sigma=(\bm I-\bm T)\bm \Sigma_{II}(\bm I-\bm T^\mathsf{T}),$
with \;$ \bm T:=\lambda_2^{-1}\lambda_1\bm v_1\bm a^\mathsf{T}\bm  v_2\bm u_2^\mathsf{T}$,\;\; $\bm \Sigma_{II}:=\bm v_2\bm u_2^\mathsf{T}\bm B(\bm v_2\bm u_2^\mathsf{T})$, \;and\;  $\bm B:=\displaystyle\sum_{i=1}^l\bm v_{1i}a_i\mathbb E(\bm\xi_i\bm\xi_i^\mathsf{T})$;\\
(ii) if $\mathrm{Re}(\lambda_2)< \lambda_1/2$, then 
$$\frac{\bm X_n-n\lambda_1\bm v_1}{\sqrt{n}}\mathop{\Rightarrow}^{d} N(0,\bm \Sigma').$$
where the covariance matrix is given by\;\;
$\bm \Sigma':=\displaystyle\int_0^\infty \psi(s,\bm A)\bm B\psi(s,\bm A)^\mathsf{T}\e^{-\lambda_1s}\lambda_1ds-\lambda_1^2\bm v_1\bm v_1^\mathsf{T},$ with $\bm B$ as above and 
\mbox{$\psi(s,\bm A):=\displaystyle \e^{s\bm A}-\lambda_1\bm v_1\bm a^\mathsf{T}\displaystyle \int_0^s \e^{t\bm A}dt.$}\\

The two options on the eigenvalues correspond to: \mbox{(i) $c_1-c_2=3/2$, and (ii) $c_1-c_2<3/2$}, respectively. To explicitly calculate the covariance matrix $\bm \Sigma$ in (i) we need to find the normalized right and left eigenvectors corresponding to the second largest eigenvalue $\lambda_2$, which are given in term of the multi-type probabilities as
$$\bm v_2=\frac{c_2}{(2-c_1+c_2)(c_2-c_1-1)}\left[\displaystyle\begin{array}{c}q_1^{11}\\
						    q_1^{12}\\
						   q_1^{22}\\
						   q_2^{22}\\
						  q_2^{12}\\
						  q_2^{11}\\
						  c_1/(c_1-c_2)\\
						   (2-c_1)/(c_1-c_2)\\ 
						   -(2-c_1)/(c_1-c_2)\\\ 
						      c_2/(c_1-c_2)\\ \end{array}\right],\quad \bm u_2=\frac{1}{c_2}\left[\displaystyle\begin{array}{c}-2(2-c_1)\\
						    c_1+c_2-2\\
						   2\\
						   2\\
						  c_1+c_2-2\\
						  -2(2-c_1)\\
						  c_1-2\\
						   1\\ 
						   1\\ 
						    c_1-2\\ \end{array}\right].$$
Computing the matrix $\bm B=(c_1-c_2-2)^{-1}[b_{i,j}]_{1\le i,j\le 10}$ gives lengthy expression for its entries
\begin{eqnarray*}
b_{1,1}\!\!\!&=&\!\!\!-\frac{q_1^{11}}{3}(10q_2^{11}-8q_1^{11}q_2^{11}+5q_2^{12}-4q_1^{11}q_2^{12}), \;
b_{1,2}=q_1^{11}(2q_2^{11}+q_2^{12})q_1^{12}, \\
b_{1,3}\!\!\!&=&\!\!\!-\frac{2q_1^{11}}{3}(2q_2^{11}+q_2^{12})(-1+q_1^{11}+q_1^{12}),\;
b_{1,4}=-\frac{q_1^{11}}{6}(2q_2^{11}+q_2^{12})(-4+2q_1^{11}-q_1^{12}),\\
b_{1,5}\!\!\!&=&\!\!\!-\frac{q_1^{11}}{3}(2q_2^{11}+q_2^{12})q_1^{12},\;  b_{16}=0,\\ 
b_{1,7}\!\!\!&=&\!\!\!-\frac{q_1^{11}}{3}(-2+2q_1^{11}+q_1^{12})q_2^{12},\;
b_{1,8}=-\frac{2q_1^{11}}{3}(-2+2q_1^{11}+q_1^{12})q_2^{11}, \\ 
b_{1,9}\!\!\!&=&\!\!\!\frac{q_1^{11}}{6}(-2+2q_1^{11}+q_1^{12})(2q_2^{11}-q_2^{12}),\;
b_{1,10}=\frac{q_1^{11}}{3}(-2+2q_1^{11}+q_1^{12})q_2^{12}\\
b_{2,1}\!\!\!&=&\!\!\!q_1^{12}(2q_2^{11}+q_2^{12})q_1^{11}, \;
b_{2,2}=-\frac{q_1^{12}}{3}(10q_2^{11}+5q_2^{12}-4q_1^{12}q_2^{11}-2q_1^{12}q_2^{12})\\
b_{2,3}\!\!\!&=&\!\!\!-\frac{q_1^{12}}{3}(2q_2^{11}+q_2^{12})(-1+q_1^{11}+q_1^{12}), \;
b_{2,4}=-\frac{q_1^{12}}{3}(2q_2^{11}+q_2^{12})(2q_1^{11}-2-q_1^{12}),\\
 &\cdots&\\
 b_{10, 7}\!\!\!&=&\!\!\! -\frac{q_2^{12}}{6}(-2+2q_1^{11}+q_1^{12})(2q_2^{11}+3q_2^{12}), \;
 b_{10, 8}=-\frac{q_2^{11}}{6}(-2+2q_1^{11}+q_1^{12})(-2+2q_2^{11}+3q_2^{12})\\
b_{10,9}\!\!\!&=&\!\!\!0,\; q_{10,10}=-\frac{1}{2}(-2+2q_1^{22}+q_1^{12})(q_2^{12}-2+2q_2^{11})
\end{eqnarray*}
Further lengthy and cumbersome linear algebra (computed using Maple) provides the given entries for the variances and covariances of different types of cherries in $\bm \Sigma$ as claimed. 

Calculating the matrix $\bm \Sigma'$ in (ii) is even more involved, due to its integral expressions, and can not be made to simplify other than in some very special cases.\\
\end{proof}

\begin{rem}
The results above are consistent with our calculations of asymptotics for the variances of the number of cherries in Proposition~\ref{prop:varcherries}. When $c_1-c_2<3/2$ implies $c_1-c_2-1<1/2$ and $2(c_1-c_2-1)<1$, and the individual variances are $\rm{O}(n)$. When $c_1-c_2=3/2$ the additional factor $\ln{n}$ comes from covariances in numbers of different types of cherries.
\end{rem}

The asymptotic strong law allows us to approximate unknown multi-type probabilities for ERM trees with a large number of leaves using counts of different types of cherries on the tree.

\begin{cor}\label{cor:ERMinferrence}
If the proportion of different types of cherries in a multi-type ERM tree is given by ${\bm x_n}=(\bm X_{n1}/n, \ldots, \bm X_{n6}/n)$
and the number of leaves $n$ in the tree is large, one can approximately recover the multi-type probabilities of the model to be 
$$q_1^{11}=\frac{\bm x_{1}}{\bm x_{1}+\bm x_{2}+\bm x_{3}},\, q_1^{12}=\frac{\bm x_{2}}{\bm x_{1}+\bm x_{2}+\bm x_{3}},\, q_1^{22}=\frac{\bm x_{3}}{\bm x_{1}+\bm x_{2}+\bm x_{3}},$$
$$q_2^{22}=\frac{\bm x_{4}}{\bm x_{4}+\bm x_{5}+\bm x_{6}},\, q_1^{12}=\frac{\bm x_{5}}{\bm x_{4}+\bm x_{5}+\bm x_{6}},\, q_1^{22}=\frac{\bm x_{6}}{\bm x_{4}+\bm x_{5}+\bm x_{6}},$$
as long as the total number of cherries with branch-point of type 1 and of type 2 are non-zero.
\end{cor}

This result is completely intuitive from a law of large numbers perspective: the multi-type probabilities for having a branch-point of type $ij_1j_2$ are given by the limiting fraction of cherries of type $ij_1j_2$. Our results on the variability of the number of cherries allows one to make a more precise statement about the error one is making using such an approximation when the number of leaves is finite.

In the standard Markov propagation model on trees the probabilities for the types of two leaves attaching to the same branch-point are independent. These are given by a stochastic transition matrix $\bm S=[s_{ij}]_{i,j\in\{1,\ldots,k\}}$ where $s_{ij}$ is the probability that a leaf of type $j$ will attach to a type $i$. In our notation this gives probabilities  $q_i^{j_1j_2}=2s_{ij_1}s_{ij_2}, \; j_1<j_2$ and $q_i^{jj}=s_{ij}^2$. Reconstruction of types for Markov propagation models has been extensively studied (see \cite{MossSteel1}). We only illustrate how the information on cherries can be used as a proxy to determine whether robust reconstruction is possible or not. Without going into all the details we recall that `reconstruction problem is solvable' if there exist two different types which when used at the root of the tree propagate asymptotically different  distributions (measured by total variation) on the leaves of the tree. This 
 roughly means that the leaf types contain a non-vanishing amount of information on the type of the root of the tree as the number of leaves $n\rightarrow\infty$. A key result (\cite{MosselS2}) then states that  on a binary tree the reconstruction problem is solvable when $\lambda_2>1/\sqrt{2}$, where $\lambda_2$ is the second largest eigenvalue of the propagation matrix $\bm S$. When $k=2$ this condition becomes $|s_{11}+s_{22}-1|>1/\sqrt{2}$ which, using Corollary~\ref{cor:ERMinferrence}, is equivalent to 
 $$\big|\sqrt{v_{1,1}/(v_{1,1}+v_{1,2}+v_{1,3})}+\sqrt{v_{1,4}/(v_{1,4}+v_{1,5}+v_{1,6})}-1\big|>1/\sqrt{2}.$$

\

\section{Ancestral tree of a multi-type birth-death process}

Consider a random tree with edge lengths, constructed from an originating node, using a pure birth process. By rescaling time one can relate any such tree to one whose birth rate is 1, even when the rate is time varying. The distribution of this tree is called 'Yule' tree (first considered in the biological context by \cite{Yule}), and has been used extensively as a null model in investigating speciation process. This is due to the fact that its distribution is precisely that of the ancestral tree reconstructed from any birth-death branching process with constant rates (\cite{Neeetal}) - an {\it ancestral tree} is obtained from a full tree of the process by pruning away all the branches without any extant species.
When the branch lengths of a Yule tree are ignored (given the same length) this produces the uniform distribution on ranked tree shapes (a {\it ranked tree} is one in which the order of branching events matters) with labelled tips, and when the ranking is also ignored it produces the (single-type) ERM distribution on binary trees (\cite{Aldous1}).

We consider a multi-type version of this tree obtained as the ancestral tree reconstructed  from a multi-type birth-death process.  
Let $\bm Z=(\bm Z(t))_{t\ge 0}$ denote a multi-type birth-death process on $\mathcal{K}=\{1,\dots,k\}$ types, whose coordinates provide the count of different types in the population $\bm Z(t)=(Z_1(t),\dots, Z_k(t))$. Let $T>0$ and let $\mathcal Z$ denote the full tree of $(\bm Z(t))_{0\le t\le T}$. Let $\mathcal W$ denote the ancestral tree obtained by pruning away all lineages of $\mathcal Z$ which do not have any extant lineages at time  $T$ (the law of $\mathcal W$ depends on $T$ but for simplicity we omit $T$ from its notation). An illustration of an ancestral tree associated with a multi-type birth-death process is shown in Figure~\ref{fig:birthdeathtrees}. Let $\bm W=(\bm W(t))_{0\le t\le T}, \bm W(t)=(W_1(t),\dots, W_k(t))$ denote the population size process of the ancestral tree $\mathcal W$ (clearly we have $\forall i$, $\forall t\in[0,T]$: $W_i(t)\le Z_i(t)$). We call $\bm W$ the {\it reconstructed ancestral process} of $\bm Z$ and derive its law, which turns out to be a multi-type pure birth process with time varying rates and an added ability to switch types along a single lineage.

\begin{figure}
\centering
{\includegraphics[width=15cm]{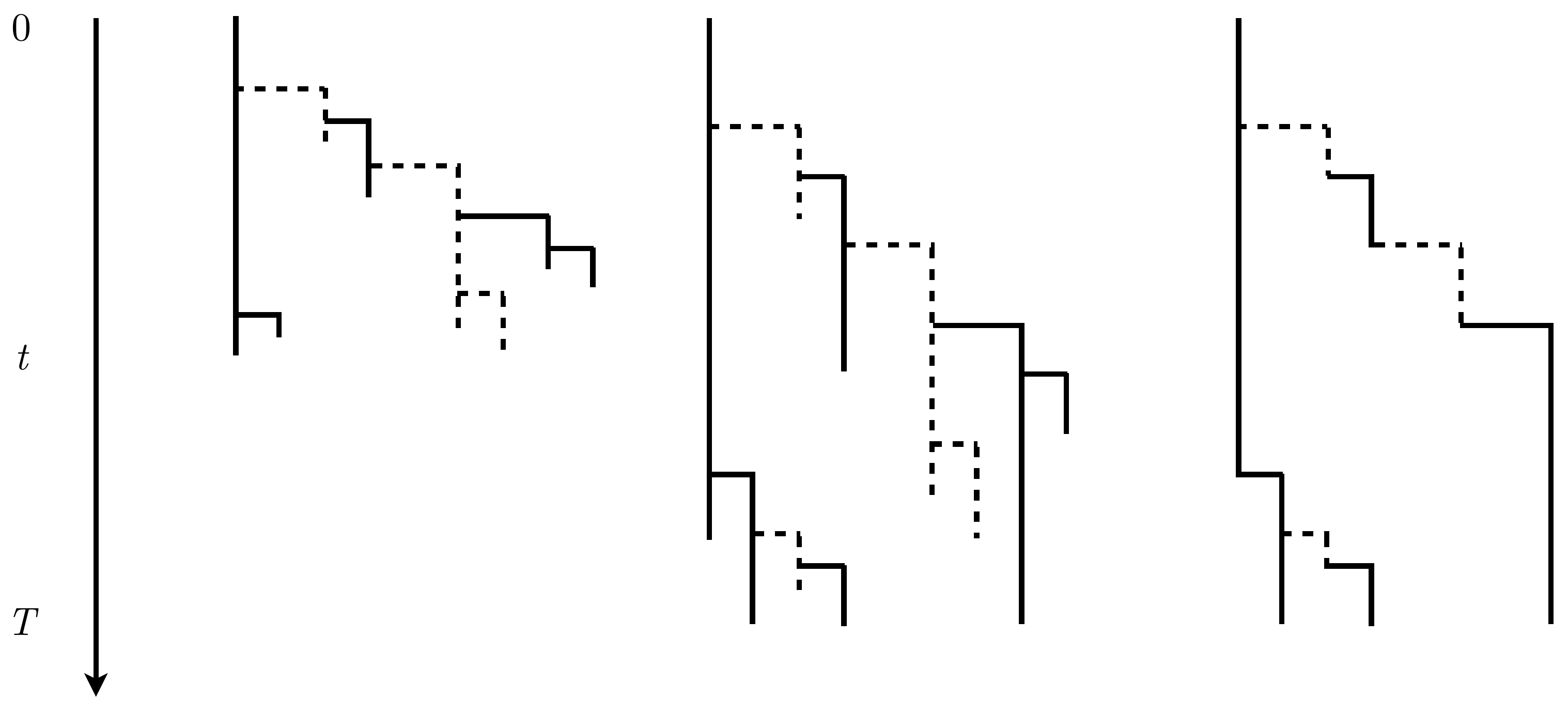}}\\
\caption{(from left to right) A tree of a two-type birth-death process; the tree of the same birth-death process until time $T$; the ancestral tree associated with the  process surviving to $T$.}\label{fig:birthdeathtrees}
\end{figure}

\begin{lem}\label{lem:ancestral-markovian}
The reconstructed ancestral process $\bm W$ of $\bm Z$ is a Markov process.
\end{lem}

\begin{proof} In the event that $\bm Z(T)=0$ there is nothing to prove, so we consider $\bm W$ on the event $\bm Z(T)\neq 0\Leftrightarrow \bm W(0)\neq 0$ (and $\bm W(T)\neq 0$ as well).

For any $n\ge 1$ let $0\leq t_0\leq t_1\leq\cdots\leq t_n\leq T$, we denote the joint distribution of $\bm W$ at these times by 
$$P_{t_0;t_1,\ldots,t_n}(\bm z_0;\bm w_1,\ldots,\bm w_n)=\P\left[ \bm W({t_j})=\bm w_j,\,1\le j\le n\,\big|\,\bm Z(t_0)=\bm z_0\right].$$
We first show, by induction, that $\forall n\geq 1$
\begin{equation}
\label{eq:markovian-induc}
P_{t_0;t_1,\ldots,t_n}(\bm z_0;\bm w_1,\ldots,\bm w_n)=
P_{t_0;t_1,\ldots,t_{n-1}}(\bm z_0;\bm w_1,\ldots,\bm w_{n-1})\frac{
P_{t_0;t_{n-1},t_n}(\bm z_0;\bm w_{n-1},\bm w_n)
}{
P_{t_0;t_{n-1}}(\bm z_0;\bm w_{n-1})
}.
\end{equation}
This is evident for $n=2$. Assume the equation 
is true $\forall i\le n-1$ with $n>2$. Notice that
\begin{equation}
\label{eq:markovian-cond-0}
P_{t_0;t_1,\ldots,t_n}(\bm z_0;\bm w_1,\ldots,\bm w_n)=\sum_{\bm z_1\ge \bm w_1}\P[\bm Z({t_1})=\bm z_1|\bm Z({t_0})=\bm z_0]P_{t_1;t_1,\ldots,t_n}(\bm z_1;\bm w_1,\ldots,\bm w_n).
\end{equation}
The branching property of the birth-death process $\bm Z$ guarantees independence of its subtrees originating from non-overlapping subsets of individuals present at any time $t_1$. Since all individuals surviving at time $T$ must be descendants of the process $\bm W$, we have  
\begin{eqnarray}
\label{eq:markovian-cond-1}
P_{t_1;t_1,\ldots,t_n}(\bm z_1;\bm w_1,\ldots,\bm w_n)&=&\P\left[ \bm W({t_j})=\bm w_j,\,1\le j\le n\,\big|\,\bm Z({t_1})=\bm z_1\right]\nonumber\\
&=&C_{\bm z_1,\bm w_1}\P\left[ \bm W({t_j})=\bm w_j,\,1\le j\le n\,\big|\,\bm Z({t_1})=\bm w_1\right]p_{\bm z_1-\bm w_1}^{\bm 0}(t_1,T)\nonumber\\
&=&C_{\bm z_1,\bm w_1}P_{t_1;t_1,\ldots,t_n}(\bm w_1;\bm w_1,\ldots,\bm w_n)p_{\bm z_1-\bm w_1}^{\bm 0}(t_1,T)
\end{eqnarray}
where $C_{\bm z_1,\bm w_1}$ denotes the combinatorial number of distinct ways of choosing $\bm w_1$ out of $\bm z_1$ individuals, and $p_{\bm z}^{\bm 0}(t,T)=\P[\bm Z(T)=0|\bm Z(t)=\bm z]$ is the extinction probability  by time $T$ of the process $\bm Z$ started at time $t$ with $\bm Z(t)=\bm z$.

Given $\bm Z({t_1})=\bm w_1$, the process $(\bm Z(t))_{t\geq t_1}$ is the sum of birth-death  processes defined by subtrees $\{\mathcal T^{(i)}\}, i=1,\ldots,|\bm w_1|$, originated by one of each of the $|\bm w_1|$ individuals at time $t_1$. We may assume that each $\mathcal T^{(i)}$ is started by an individual of type $\tau^{(i)}$, where $\tau^{(1)},\ldots,\tau^{(|\bm w_1|)}$ is some ordering of the $|\bm w_1|$ surviving originator types. Probability for the surviving lineages is
\begin{eqnarray*}
P_{t_1;t_1,\ldots,t_n}(\bm w_1;\bm w_1,\ldots,\bm w_n)
&=&\P\left[ \bm W({t_j})=\bm w_j,\,1\le j\le n\,\big|\,\bm Z({t_1})=\bm w_1\right]\\
&=&\P\Big[ \bm W({t_j})(\mathcal T^{(i)})\neq 0\, \forall i,\;\, \sum_{i=1}^{|\bm w_1|} \bm W({t_j})(\mathcal T^{(i)})=\bm w_j, \,\forall 2\le j\le n\Big]
\end{eqnarray*}
where $\bm W(t)(\mathcal T^{(i)})$ denotes the number of individuals of $\mathcal T^{(i)}$ at time $t$ which have a surviving lineage at time $T$. Since the subtrees $\mathcal T^{(i)}$ are independent
\begin{equation}
\label{eq:markovian-ind-trees}
P_{t_1;t_1,\ldots,t_n}(\bm w_1;\bm w_1,\ldots,\bm w_n)=
\displaystyle \sum_{\substack{\forall 2\le j\le n,\; (\bm w_j^{(i)})_{1\le i\le|\bm w_1|}:\\\bm w_j^{(i)}>0,\; \sum_{i=1}^{|\bm w_1|}\bm w_j^{(i)}  =  \bm w_j }}
\prod_{i=1}^{|\bm w_1|}P_{t_1;t_2,\ldots,t_n}(\bm e_{\tau^{(i)}};\bm w^{(i)}_2,\ldots,\bm w^{(i)}_n),
\end{equation}
where $\bm e_i$ denotes the unit $k$-dimensional vector whose $i$-th coordinate is 1 and  all other coordinates are 0, and the summation is over all possible decompositions of $\bm w_j$ into vectors $(\bm w_j^{(i)})_{i=1,\ldots,|\bm w_1|}$ with all nonzero coordinate values, for each $j=2,\dots, n$.
By the inductive hypothesis (\ref{eq:markovian-induc}) for $n-1$, the probabilities in the product on the right side are equal to
\begin{eqnarray*}
\!\!\!\!\!\!&&\!\!\!\! P_{t_1;t_2,\ldots,t_n}(\bm e_{\tau^{(i)}};\bm w^{(i)}_2,\ldots,\bm w^{(i)}_n)\\
\!\!\!&&=\displaystyle
P_{t_1;t_2,\ldots,t_{n-1}}(\bm e_{\tau^{(i)}};\bm w^{(i)}_2,\ldots,\bm w^{(i)}_{n-1})
\frac{
P_{t_1;t_{n-1},t_n}(\bm e_{\tau^{(i)}};\bm w^{(i)}_{n-1},\bm w^{(i)}_n)
}{
P_{t_1;t_{n-1}}(\bm e_{\tau^{(i)}};\bm w^{(i)}_{n-1})
}\\
\!\!\!&&=
P_{t_1;t_2,\ldots,t_{n-1}}(\bm e_{\tau^{(i)}};\bm w^{(i)}_2,\ldots,\bm w^{(i)}_{n-1})
\P[\bm W({t_n})=\bm w_n^{(i)}|\bm W({t_{n-1}})=\bm Z({t_{n-1}})=\bm w_{n-1}^{(i)}]
\end{eqnarray*}
where the last equality follows from (\ref{eq:markovian-cond-0}) and (\ref{eq:markovian-cond-1}) since
\begin{eqnarray*}
\!\!\!\!\!\!&&\!\!\!\! P_{t_1;t_{n-1},t_n}(\bm e_{\tau^{(i)}};\bm w^{(i)}_{n-1},\bm w^{(i)}_n)\\
\!\!\!\!&&=\,\P[\bm W({t_{n-1}})=\bm w_{n-1}^{(i)}, \bm W({t_n})=\bm w_n^{(i)}|\bm Z({t_1})=\bm e_{\tau^{(i)}}]\\
\!\!\!\!\!\!&&=\!\!\!\!\sum_{\bm z_{n-1}\ge \bm w_{n-1}} \!\!\!\!\!\!\!\!
\P[Z({t_n-1})=\bm z_{n-1}|\bm Z({t_1})=\bm e_{\tau^{(i)}}]C_{\bm z_{n-1};\bm w_{n-1}^{(i)}} p^{\bm 0}_{\bm z_{n-1}-\bm w_{n-1}}(t_{n-1},T) \\
&&\qquad\qquad\qquad\qquad\qquad\qquad\qquad\qquad\qquad\qquad\qquad\cdot\P[\bm W({t_n})=\bm w_n^{(i)}|\bm W({t_{n-1}})=\bm Z({t_{n-1}})=\bm w_{n-1}^{(i)}]\\ \\
\!\!\!\!&&=\,\P[\bm W({t_n})=\bm w_n^{(i)}|\bm W({t_{n-1}})=\bm Z({t_{n-1}})=\bm w_{n-1}^{(i)}]P_{t_1;t_{n-1}}(\bm e_{\tau^{(i)}};\bm w^{(i)}_{n-1}).
\end{eqnarray*}

As the first factor on the right side above does not depend on $(\bm w_n^{(i)})_{i=1,\ldots,|\bm w_1|}$ the sum in (\ref{eq:markovian-ind-trees}) may be split into outer sums, over $2\le j\le n-1$, and an inner sum, over $j=n$ that is equal to
$$
\sum_{\substack{(\bm w_n):\bm w_n^{(i)}>0,\\ \sum_{i=1}^{|\bm w_1|}\bm w_n^{(i)}  =  \bm w_n}}
\prod_{i=1}^{|\bm w_1|}\P[\bm W({t_n})=\bm w_n^{(i)}|\bm W({t_{n-1}})=\bm Z({t_{n-1}})=\bm w_{n-1}^{(i)}].
$$
By the same argument using splitting over independent subtrees, but this time splitting the individuals at time $t_{n-1}$ into subsets of sizes $(\bm w_{n-1}^{(i)})_{i=1,\ldots,|\bm w_1|}$, we can show that this sum contributes to the outer sums a factor of $$\P[\bm W({t_n})=\bm w_n|\bm W({t_{n-1}})=\bm Z({t_{n-1}})=\bm w_{n-1}]=
\frac{
P_{t_0;t_{n-1},t_n}(\bm z_0;\bm w_{n-1},\bm w_n)
}{
P_{t_0;t_{n-1}}(\bm z_0;\bm w_{n-1})
},$$
where the last equality follows 
again from equations (\ref{eq:markovian-cond-0}) and (\ref{eq:markovian-cond-1}), and combining with the outer sums in (\ref{eq:markovian-ind-trees}) implies
$$
P_{t_1;t_1,\ldots,t_n}(\bm w_1;\bm w_1,\ldots,\bm w_n) = P_{t_1;t_1,\ldots,t_{n-1}}(\bm w_1;\bm w_1,\ldots,\bm w_{n-1})
\frac{
P_{t_0;t_{n-1},t_n}(\bm z_0;\bm w_{n-1},\bm w_n)
}{
P_{t_0;t_{n-1}}(\bm z_0;\bm w_{n-1})
},
$$
as wanted. By using once again equations (\ref{eq:markovian-cond-0}) and (\ref{eq:markovian-cond-1}), this becomes equation (\ref{eq:markovian-induc}) for step $n$. Equation (\ref{eq:markovian-induc}) may be written in terms of conditional probabilities as
\begin{eqnarray*}
\P\left[\bm W({t_n})=\bm w_n\,\big|\,\bm W({t_j})=\bm w_j,\,1\le j\le n-1,\,\bm Z({t_0})=\bm z_0\right]\\
=\P\left[\bm W({t_n})=\bm w_n\,\big|\,\bm W({t_{n-1}})=\bm w_{n-1},\,\bm Z({t_0})=\bm z_0\right]
\end{eqnarray*}
which implies the Markov property for $(\bm W(t))_{t\geq0}$.
\end{proof}

\begin{prop}\label{prop:ancestral} Assume the multi-type birth-death process $\bm Z$  has birth rates $\{b_{i}^{ij}\}_{i,j\in\{1,\dots, k\}}$ ($b_{i}^{ij}$= rate at which any type $i$ gives birth to a type $j$) and death rates $\{d_i\}_{i\in\{1,\dots, k\}}$ ($d_i$= rate at which any type $i$ dies). Then, for any $T>0$, the reconstructed ancestral process $\bm W$  is a pure birth process with birth rates $\{q_i^{ij}(t)\}_{i,j\in\{1,\dots, k\}}$ ($q_{i}^{ij}$= rate at which type $i$ gives birth to type $j$) and mutation rates $\{q_i^{j}(t)\}_{i\in\{1,\dots, k\}}$ ($q_i^{j}$= rate at which type $i$ changes into type $j$) at time $t\in[0, T)$, given by 
\begin{equation}\label{eq:ancestral-rates} 
q_i^{ij}(t)=b_{i}^{ij}(1-p^{\bm 0}_{\bm e_j(t,T)})\; \forall i,j, \quad  q_i^{j}(t)=b_{i}^{ij}(1-p^{\bm 0}_{\bm e_j(t,T)})\frac{p^{\bm 0}_{\bm e_i(t,T)}}{1-p^{\bm 0}_{\bm e_i(t,T)}}\; \forall i\neq j
\end{equation}
where $p^{\bm 0}_{\bm e_i(t,T)}=\P[\bm Z(T)=0|\bm Z(t)=\bm e_i]$ are the extinction probabilities for $\bm Z$.
\end{prop}

\begin{rem} When there is only one type, for example $i$, this reduces to a pure birth process with time varying birth rate $b_i(1-p^{\bm 0}_{\bm e_i})$ as previously established (\cite{Neeetal}).
The extinction probabilities $\{p^{\bm 0}_{\bm e_i}(t,T)\}_{i\in\{1,\ldots,k\}}$ can be shown to satisfy a system of differential equations (\cite{M08}, \cite{J11})
$$\frac{d p_{\bm e_i}^{\bm 0}(t,T)}{dt}=d_i-(\sum_{j=1}^kb_{i}^{ij}+d_i)p_{\bm e_i}^{\bm 0}(t,T)+\sum_{j=1}^k b_i^{i j}p_{\bm e_i}^{\bm 0}(t,T)p_{\bm e_j}^{\bm 0}(t,T),\quad  i=1,\ldots,k.
$$
\end{rem}

\begin{proof}
By Lemma~\ref{lem:ancestral-markovian}, the reconstructed ancestral process $(\bm W(t))_{t\geq0}$ is Markov, so it suffices to show that its only transitions are changes of the form $\{\bm e_i, { i=1,\ldots, k}\}$ and $\{\bm e_j-\bm e_i, { i\neq j=1,\ldots,k}\}$ and calculate their  rates. The set of possible transition changes for $\bm Z$, and the fact that $|\bm W(t)|$ is non-decreasing, imply the form of changes for $\bm W$: an addition of $\bm e_i$ occurs iff there is a birth event and both the new lineage and the parent lineage survive to $T$, an addition of $\bm e_i-\bm e_j,$ occurs iff there is a birth event and only the new lineage survives to $T$ (see Figure~\ref{fig:birthdeathtrees} for an example). 

Considering the possible values of the underlying birth-death process $\bm Z$ for a transition in $(t,t+\Delta t]$, using (\ref{eq:markovian-cond-1}), we get
\begin{eqnarray}\label{eq:ancestral-transition-birth}
&&\!\!\!\!\!\!\!\!\!\!\!\!\!\!\!\!\P[\bm { W}_{t+\Delta t}=\bm w+\bm e_j\,|\, \bm W(t)=\bm  w]\nonumber\\ 
&&=\displaystyle\frac{\displaystyle\sum_{\bm z\ge \bm w}\P[\bm W(t+\Delta t)=\bm w+\bm e_j,\bm W(t)=\bm w,\bm Z(t)=\bm z]}{\displaystyle\sum_{\bm z\ge w}\P[\bm W(t)=\bm w,\bm Z(t)=\bm z]}\nonumber\\
&&=\displaystyle\frac{\displaystyle\sum_{\bm z\ge \bm w}\P[\bm Z(t)=\bm z]C_{\bm z,\bm w}\sum_{i=1}^k \bm w_i{b}_i^{ij}\Delta t\big(\bm 1-\bm p^{\bm 0}(t+\Delta t,T)\big)^{\bm w+\bm e_i}\bm p^{\bm 0}(t+\Delta t,T)^{\bm z-\bm w}+o(\Delta t)}{\displaystyle\sum_{\bm z}\P[\bm Z(t)=\bm z]C_{\bm z,\bm w}(\bm 1-\bm p^{\bm 0}(t,T))^{\bm w}\bm p^{\bm 0}(t,T)^{\bm z-\bm w}}\nonumber\\ 
&&=\sum_{i=1}^k\bm w_i {b}_i^{ij}(1-p_{\bm e_j}^{\bm 0}(t,T))\Delta t+o(\Delta t).
\end{eqnarray}
where we used notation\; $\bm p^{\bm 0}(t,T)^{\bm w}:=\displaystyle\prod_{i=1}^k\bm p^{\bm 0}_{\bm e_i}(t,T)^{\bm w_i} $,\, $(\bm 1-\bm p^{\bm 0}(t,T))^{\bm w}:=\displaystyle\prod_{i=1}^k(1-\bm p^{\bm 0}_{\bm e_i}(t,T))^{\bm w_i} .$ \\
Similarly for $i\neq j$
\begin{eqnarray}\label{eq:ancestral-transition-mutation}
&&\!\!\!\!\!\!\!\!\!\!\!\!\!\!\!\!\P[\bm W(t+\Delta t)=\bm w+\bm e_j-\bm e_i\,|\, \bm W(t)=\bm  w] \nonumber\\
&&=\displaystyle\frac{\displaystyle\sum_{\bm z}\P[\bm Z(t)=\bm z]C_{\bm z,\bm w}\, \bm w_i{b}_i^{ij}\Delta t\big(\bm 1-\bm p^{\bm 0}(t+\Delta t,T)\big)^{\bm w+\bm e_j-\bm e_i}\bm p^{\bm 0}(t+\Delta t,T)^{\bm z-\bm w+\bm e_i}+o(\Delta t)}{\displaystyle\sum_{\bm z}\P(\bm Z(t)=\bm z)C_{\bm z,\bm w}(\bm 1-\bm p^{\bm 0}(t,T))^{\bm w}\bm p^{\bm 0}(t,T)^{\bm z-\bm w}}\nonumber\\ 
&&=\displaystyle\frac{\bm w_j {b}_i^{ij}(1-p_{\bm e_j}^{\bm 0}(t,T))p_{\bm e_i}^{\bm 0}(t,T)}{1-p_{\bm e_i}^{\bm 0}(t,T)}\Delta t+o(\Delta t).
\end{eqnarray}

Transition rates (\ref{eq:ancestral-transition-birth}) and (\ref{eq:ancestral-transition-mutation}) correspond to those of a pure birth process allowing for mutations along the lineages as claimed in (\ref{eq:ancestral-rates}).
\end{proof}

In continuous time $t\in[0,T)$ nodes of different types have different time varying {\it weights}, such that at any time the probability of a node of certain type is chosen to be the next node with a branch-point (binary or unary) is proportional to this weight. The weight of a node of type $i$ is $a_{i}(t)={q_{i}(t)}/{\sum_{\ell=1}^kq_{\ell}(t)}$ where 
$$  q_i(t)=\frac{1}{1-p^{\bm 0}_{\bm e_i}(t,T)}\Big(\sum_{j=1}^k b_{i}^{ij}(1-p^{\bm 0}_{\bm e_j}(t,T))-b_i^{ii}p^{\bm 0}_{\bm e_i}(t,T)(1-p^{\bm 0}_{\bm e_i}(t,T))\Big) $$
is the overall rate of events for type $i$. 
The probabilities of a node of type $i$ having a binary branch-point (of type $i$ and $j$) versus a unary branch-point (of type $j\neq i$) are
\begin{eqnarray*}
p_{i}^{i j}(t)&&=\frac{b_i^{i j}(1- p_{\bm e_{j}}^{\bm 0}(t,T) )( 1- p_{\bm e_{i}}^{\bm 0}(t,T))}{\sum_{\ell=1}^k b_i^{i \ell}(1- p_{\bm e_{\ell}}^{\bm 0}(t,T))-b_i^{ii}p^{\bm 0}_{\bm e_i}(t,T)(1-p^{\bm 0}_{\bm e_i}(t,T))}\; \forall i,j,\\ 
p_{i}^j(t)&&=\frac{ \displaystyle b_i^{ i j}(1-p_{\bm e_j}^{\bm 0}(t,T))p_{\bm e_{i}}^{\bm 0}(t,T)}{\sum_{\ell=1}^k b_i^{i \ell}(1- p_{\bm e_{\ell}}^{\bm 0}(t,T))-b_i^{ii}p^{\bm 0}_{\bm e_i}(t,T)(1-p^{\bm 0}_{\bm e_i}(t,T))}\; \forall i\neq j.
\end{eqnarray*}

Contrary to the single type case, it is not possible to rescale time and relate this to a Yule process with constant rates of birth and mutations, because the rate at which the time needs to be rescaled depends on the type of the node that was involved in the last branching event. This information is dependent on the randomness of the tree and is not simply a  deterministic function of time as it is in the single type case. 
Consequently, ignoring the edge lengths and possibly the ranking of branching events in these trees does not produce any logical model on multi-type discrete trees. Topologically it results in multi-type discrete trees which are no longer regular binary ones, as in addition to binary branch-points they also have unary branch-points (with the type attached being necessarily different). Figure~\ref{fig:yuletrees} illustrates obtaining such a discrete tree.

 \begin{figure}
\centering
{\includegraphics[width=10cm]{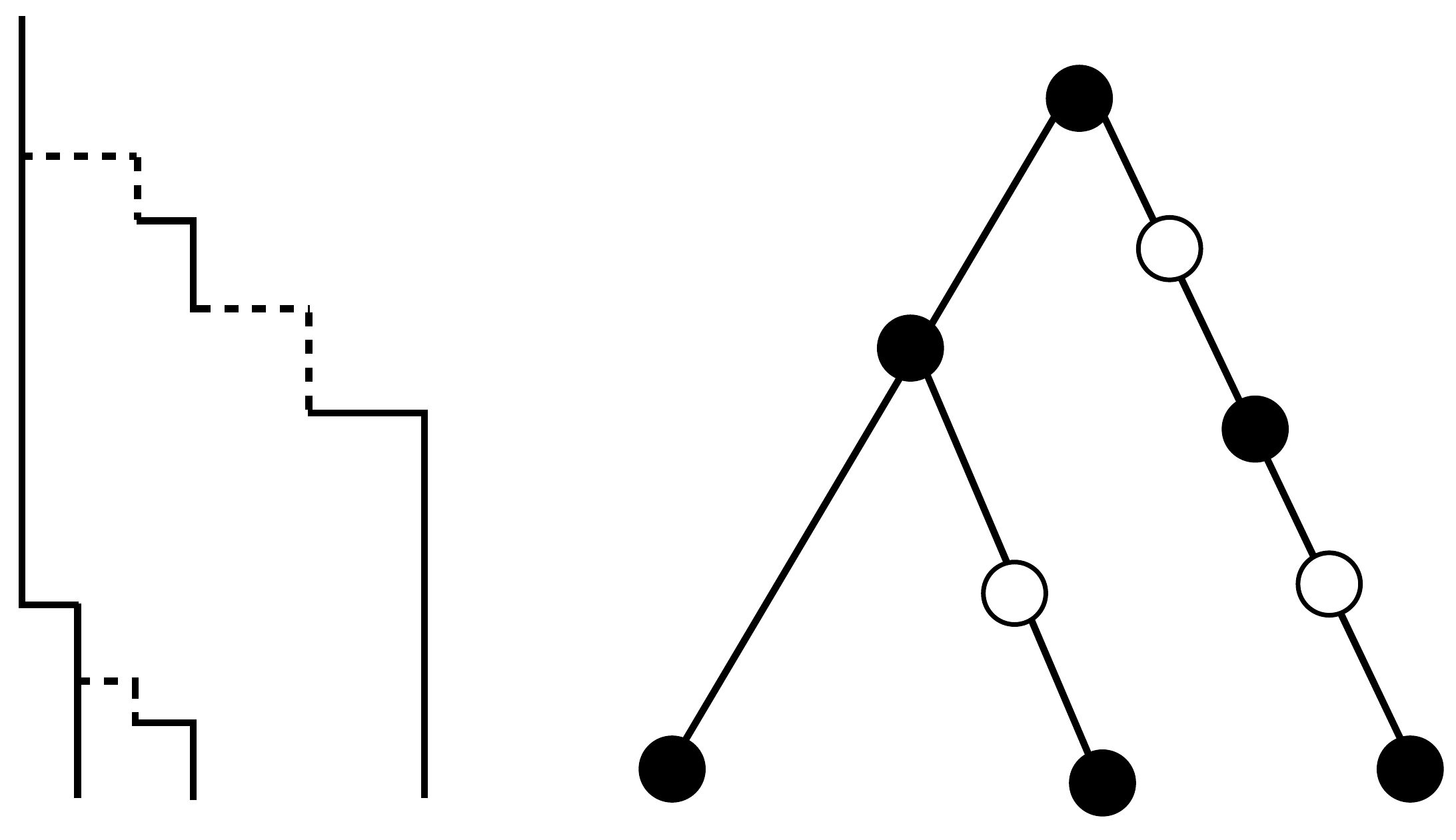}}\\
\caption{The ancestral tree from Figure~\ref{fig:birthdeathtrees} and the corresponding discrete two-type tree with branch-points and mutations, obtained by ignoring edge-lengths and ranking in the former.}\label{fig:yuletrees}
\end{figure}

However, as in the single-type case, near the present ($t\approx T$) probabilities of extinction $\bm p^0(t,T)$ are approximately zero, and birth and mutation rates in the ancestral tree are approximately constant $q_{i}^{ij}\approx b_{i}^{ij}, \forall i, j$, and  $q_i^j\approx 0, \forall i\neq j$. This allows one to infer birth rates of the process using results on constant rate multi-type Yule trees described in the next Section (see Corollary~\ref{cor:contsinferrence}). Knowing the values of lineage through time plots for different types ($\bm Z(t),0\le t\le T$)  will then allow one to also infer death rates of the process.

\begin{rem}\label{rem:fromYuletoERM}
If we consider multi-type Yule trees with mutations whose birth rates $\{q_i^{ij}\}_{i,j\in\{1,\ldots, k\}}$ and mutation rates $\{q_i^j\}_{i\neq j\in\{1,\ldots,k\}}$ are constant,  ignoring edge lengths results in a useful model on multi-type discrete trees:\; 
each node of type $i$ is chosen to be the next branch-point with probability proportional to its weight $$a_i=\frac{q_i}{\sum_{\ell=1}^k q_\ell},\; \mbox{ where }\; q_i=\sum_{j=1}^k q_i^{ij}+\sum_{\substack{j=1\\ j\neq i}}^kq_i^j;$$ once chosen the branch-point is binary with attached leaves of types $i,j$, or unary with attached leaf of type $j\neq i$, respectively, with probabilities $$p_{i}^{ij}=\frac{q_i^{ij}}{q_i}\;\forall i,j,\; \mbox{ and }\;  p_{i}^{j}=\frac{q_i^j}{q_i}\;\forall i\neq j.$$ 
The distribution of different types of cherries and pendants in the tree should provide information about its birth and mutation rates. However, the approaches for obtaining their distribution using generating functions and recursive relations (when the number of leaves is finite), as well as the Polya urn approach for their asymptotic distribution (as the number of leaves grows), are completely unwieldly. The more appropriate approach is to analyze distributions of different types of cherries and pendants in the original continuous time  trees as shown in the next Section.
\end{rem}

\section{Multi-type Yule trees with mutations}
We consider a multi-type birth process with mutations constructed using time-dependent birth rates $\{q_i^{j_1j_2}(t)\}_{i,j_1,j_2\in\{1,\ldots,k\}}$ and mutation rates $\{q_i^j(t)\}_{i\neq\in\{1,\ldots,k\}}$, and call its associated tree a {\it multi-type Yule tree with mutations}. For generality, we allow for the birth events to result in an instantaneous change of type for the parent node as well, so that birth rates for a parent node of type $i$ are indexed  in the superscript by any $j_1, j_2\in\{1,\ldots,k\}$  giving a birth event of type $ij_1j_2$ (rather than only having birth events of types $iij$ as in ancestral trees of the previous Section). Consequently, each birth event is a branch-point (with no special designation in the continuing lineages) and in order not to distinguish between different  planar embeddings we will w.l.o.g. assume that $j_1\le j_2$  (as in the multi-type ERM case). For $k$ types this model has $k^2(k+1)/2+k(k-1)$  parameters. 

Due to mutations in the model (producing unary branch-points) we need to precise a  definition of multi-type cherries and pendants in such a tree. Since the sequence of mutation events along a lineage is typically not  available in data, we will focus on the types at the topological end-points of the structure. We first let the topology of the tree be defined only by binary branch-points, while unary branch-points are ignored. The cherries and pendants are then defined in this topology as they would be in a regular binary tree. This means that the type of each cherry and each pendant is defined by the type values at the end nodes of the cherry or pendant. respectively. Figure~\ref{fig:yuletrees} illustrates a two-type Yule tree with mutations which has only one cherry of type 222 and only one pendant of type 22.  In general there are $k^2(k+1)/2$ different types of cherries (we don't differentiate between different planar embeddings of a cherry type), $k^2$ different types of pendants (sequence of mutations along a lineage can revert to the original type), and $k$ different types of leaves. 

\subsection{Moments of the number of different types of cherries and pendants}\label{subsec:Yulemoments}

For a multi-type Yule tree with mutations, we let $N_1(t), \ldots, N_k(t)$ denote the number of leaves of types $1,\ldots, k$, respectively, at time $t$. Let $C_{i}^{j_2j_2}(t)$ denote the number of cherries of type $ij_1j_2$, and $L_i^j(t)$ the number of pendants of type $ij$ at time $t$. We next consider their means, which are relatively straightforward, although quite complicated, to calculate.

\begin{lem}\label{lem:meannudependant-eq}
Let $\bm {\nu}(t)=(\nu_1(t),\ldots,\nu_k(t))$ be the vector of leaf means, 
$\nu_i(t):=\mbb E[N_i(t)], \forall i$. Then, $\forall t\ge 0$
$$\frac{d\bm{\nu}(t)}{dt}=\bm B(t)\bm{\nu}(t)
$$ 
where $\bm B(t)$ is the $k\times k$ matrix with entries
$$[\bm B(t)]_{\ell_1,\ell_2}=\left\{\begin{array}{lc} \displaystyle q_{\ell_1}^{\ell_1\ell_1}(t)-\sum_{\substack{i\leq j \\ i,j\neq\ell_1}} q_{\ell_1}^{ij}(t)-\sum_{i\neq\ell_1} q_{\ell_1}^{i}(t) &\mbox{ when } \ell_1=\ell_2\\ \\
\displaystyle 2 q_{\ell_2}^{\ell_1\ell_1}(t)+ q_{\ell_2}^{\ell_1}(t)+\sum_{j < \ell_1} q_{\ell_2}^{j\ell_1}(t)+\sum_{j > \ell_1} q_{\ell_2}^{\ell_1 j}(t)& \mbox{ when } \ell_1\neq \ell_2.\end{array}\right.$$
\end{lem}
\begin{proof}
The matrix formulation is equivalent to the claim that each $\nu_{\ell}(t)$ for $1\le \ell\le k$ satisfies
$$\frac{d\nu_{\ell}(t)}{dt}=\sum_{i\neq\ell} \Big(2 q_{i}^{\ell\ell}(t)+ q_{i}^{\ell}(t)+\sum_{j<\ell} q_{i}^{j\ell}(t)+\sum_{j>\ell} q_{i}^{\ell j}(t)\Big)\nu_i(t) +\Big( q_{\ell}^{\ell\ell}(t)-\sum_{\substack{i\leq j\\ i,j\neq\ell}} q_{\ell}^{ij}(t)-\sum_{i\neq\ell} q_{\ell}^{i}(t)\Big) \nu^{\ell}(t).$$ 
To see why this is true, observe that in the time interval $(t,t+\Delta t)$ the number of leaves of type $\ell$ increases by 2 iff we have a birth event of type $i\ell\ell$ for some $i\neq\ell$. It increases by 1 iff we have a birth event of type $ij\ell$ for some $i,j\neq\ell$ or for $i=j=\ell$, or if we apply a change of type $i\ell$ for some $i\neq\ell$. The number of leaves of type $\ell$ decreases by 1 only by having birth events of types $\ell ij$ or by having type changes $\ell i$ for $i,j\neq\ell$.
\end{proof}

\begin{rem}\label{rem:meannudependant}
 If $\bm B(t)$ in Lemma~\ref{lem:meannudependant-eq} is such that it commutes with $\int_0^t\bm B(\tau) dt\tau$ $\forall t\geq0$, then the vector of leaf means can be given explicitly as
$$\bm{\nu}(t)=\int_0^t\exp\{\bm B(\tau)\}d\tau\,\bm \nu(0), \;\mbox{ and }\; \bm{\nu}(t)=\exp\{\bm Bt\}\,\bm \nu(0)
$$
if $\bm B(t)$ is a constant (time-independent) matrix $\bm B$.\\
\end{rem}

\indent Let $\rho(t):=\sum_{i=1}^k\nu_i(t)$. By adding up counts for all different leaves we obtain the following.
\begin{cor}\label{cor:meanrho}
Assume $\bm B(t)$ commutes with $\int_0^t\bm B(\tau) d\tau$ $\forall t\geq0$, then 
$$\rho(t)=\bm 1^\mathsf{T}\, \int_0^t\exp\{\bm B(\tau)\}d\tau\, \bm \nu(0),  \;\mbox{ and }\;\; \rho(t)=\bm 1^\mathsf{T}\exp\{\bm Bt\}\bm \nu(0)\; \mbox{ if }\bm B(t)\equiv \bm B\, \forall t\ge 0.$$
\end{cor}

\vspace{5mm}
We next give the mean number of cherries whose branch-point is of type $\ell$. The mean number of cherries with branch-points of other types can be obtained analogously.

\begin{prop}\label{prop:meanmudependant}
Let $\bm \mu_{\ell}(t)=(\mu_{\ell}^{11}(t),\ldots,\mu_{\ell}^{kk}(t))$ be the vector of cherry means \mbox{$\mu_{\ell}^{ij}(t):=\mbb E[C_{\ell}^{11}(t)]$} of types $\ell ij$, for $i\leq j\in\{1,\ldots, k\}$. Then, $\forall t\ge 0$
$$\frac{d\bm \mu_\ell(t)}{dt}= \bm A_\ell(t)\bm \mu_\ell(t)+\bm {q}_{(\ell)}(t) \nu_\ell(t),$$
where $$\bm {q}_{(\ell)}(t):=[q_{\ell}^{11}(t),q_{\ell}^{12}(t), \ldots, q_{\ell}^{kk}(t)]^\mathsf{T}$$
and $\bm A_\ell(t)$ is a ${{k+1}\choose 2}\times {{k+1}\choose 2}$ matrix with entries
$$[\bm A_\ell(t)]_{\ell ij,\ell m n}=\left\{\begin{array}{ll} -(q_{i}(t)+q_{j}(t)), &\mbox { when } (m,n)=(i,j)\\ \delta_{m,i} q_{n}^i(t)+ \delta_{n,i}q_{m}^i(t), & \mbox { when } (m,n)\neq (i,j),\, i= j\\   \delta_{m,i}q_{n}^j(t)+ \delta_{m,j}q_{n}^i(t)+ \delta_{n,i}q_{m}^j(t)+ \delta_{n,j}q_{m}^i(t),& \mbox { when } (m,n)\neq (i,j),\, i\neq j.\end{array}\right.$$
where, for $i\in\{1,\ldots,k\}$, $q_i(t)$ is the overall rate of events occurring to a lineage of type $i$
$$q_i(t):=\sum_{j\le \ell}q_i^{j\ell}+\sum_{j\neq i}q_i^j,$$ 
and where entries in the matrix $\bm A_\ell(t)$ are ordered in a consistent way with that of types in the vectors $\bm\mu_\ell(t)$, $\bm {q}_{(\ell)}(t)$.
\end{prop}

\begin{proof}
We show that each $\mu_{\ell}^{ ij}(t)$ satisfies the following differential equation when $i=j$
$$
\frac{d\mu_{\ell}^{ij}(t)}{dt}=\!\!\!\!\!\sum_{\substack{m\leq n\\ (m,n)\neq(i,i)}}\!\!\!\! \big( \delta_{m,i} q_{n}^i(t)+ \delta_{n,i} q_{m}^i(t)\big)\mu_{\ell}^{ m n}-2 q_{i}(t)\mu_{\ell}^{ ii}(t)+ q_{\ell} ^{ii}(t)\nu_\ell(t),$$
and when $i\neq j$ it satisfies
$$
\frac{d\mu_{\ell}^{ij}(t)}{dt}=\!\!\!\!\!\sum_{\substack{m\leq n\\ (m,n)\neq(i,j)}}\!\!\!\! \big( \delta_{m,i}q_{n}^j(t)+ \delta_{m,j}q_{n}^i(t)+ \delta_{n,i}q_{m}^j(t)+ \delta_{n,j}q_{m}^i(t)\big)\mu_{\ell}^{ m n}(t)-(q_{i}(t)+q_{j}(t))\mu_{\ell}^{ ij}(t)+q_{\ell} ^{ij}(t)\nu_\ell(t).
$$
This can be seen from the fact that the number of cherries of type $\ell ij$ will increase by 1 iff a cherry of type $\ell ij$ is added by a birth event to a lineage of type $\ell$, or there is a mutation along a lineage of a cherry which from a cherry of some different type produces a cherry of type $\ell ij$.  The number of cherries of type $\ell ij$ will decrease by 1 iff there is a mutation along a lineage of a type $\ell ij$ cherry, or there is a birth event along one of its lineages producing a cherry of some different type.
\end{proof}

\begin{rem}\label{rem:Adiagdominant}
The matrix $\bm A_\ell(t)$ is diagonally dominant by columns, as:\, in a column $\ell m n$ every rate of the form $q_{m}^i(t)$ and every rate of the form $ q_{n}^i(t)$ appears exactly once (when $n=m$ each one appears twice) and the sum of these rates is less than or equal to $q_{n}(t)+q_{m}(t)$. We will use this fact in upcoming proofs.
\end{rem} 

\begin{prop}\label{prop:meangamma}
Let $\bm \gamma(t)=(\gamma_1^{1}(t),\ldots, \gamma_k^{k}(t))$ be the vector of pendant means \mbox{$\gamma_i^j(t):=\mbb E[L_i^j(t)]$}. Then, $\forall t\ge 0$
$$\frac{d\bm \gamma(t)}{dt}=\bm C(t)\bm \gamma(t) +\bm U(t)\bm \mu(t), $$
where $\bm C(t)$ is a $k^2\times k^2$ matrix with entries
$$[\bm C(t)]_{\ell m,ij}=\left\{\begin{array}{rl} -q_{m}(t)& \mbox { when } (\ell,m)=(i,j)\\   q_{j}^m(t)& \mbox { when } \ell=i,\, m\neq j\\ 0 &\mbox{ otherwise. }\end{array}\right.$$ and $\bm U(t)$ is a $k^2\times {{k+1}\choose 2}$ matrix with entries
$$[\bm U(t)]_{\ell m,\ell' i j}=\left\{\begin{array}{rl}\displaystyle 2\sum_{j_1\leq j_2}q_{m}^{j_1j_2}(t)& \mbox { when } \ell=\ell',\, m=i=j\\   \displaystyle \sum_{j_1\leq j_2}\ q_{i}^{j_1j_2}(t)& \mbox { when } \ell=\ell',\, m=j> i\\   \displaystyle \sum_{j_1\leq j_2}\  q_{j}^{j_1j_2}(t)& \mbox { when } \ell=\ell',\, m=i<j\\  0 &\mbox{ otherwise}.\end{array}\right.$$ 
\end{prop}

\begin{proof}
We show that each $\gamma_{\ell}^{ m}(t)$ satisfies
\begin{eqnarray*}
\frac{d\gamma_{\ell}^{ m}(t)}{dt}= \sum_{j\neq m}q_{j}^m(t)\gamma_{\ell}^{ j}(t)-q_{m}(t)\gamma_{\ell}^{ m}(t)+\sum_{i<m}\Big(\sum_{j_1\leq j_2}q_{i}^{j_1 j_2}(t)\Big)\mu_{\ell}^{ i m}(t)\\+\sum_{i>m}\Big(\sum_{j_1\leq j_2}q_{i}^{j_1 j_2}(t)\Big)\mu_{\ell}^{ m i}(t)+2\sum_{j_1\leq j_2} q_{m}^{j_1j_2}(t)\mu_{\ell}^{ m m}(t).\end{eqnarray*}
To see this, observe that the number of pendant edges of type $\ell m$ will increase by 1 if a mutation of type $jm$ occurs on a pendant edge of type $\ell j$. Also, it will increase by 1 if a birth event happens adding any cherry of type $ij_1j_2$ to any cherry of type $\ell i m$ $(i\neq m)$, or if a cherry of type $m j_1j_2$ is added to a cherry of type $\ell mm$. The number of pendant edges of type $\ell m$ will decrease by 1 iff any birth event adding a cherry or mutation occurs on a pendant edge of type $\ell m$. 
\end{proof}

\subsection{Long time asymptotics for the number of cherries and pendants}\label{subsec:Yulelimits}

We next consider what happens to the tree structure of the multi-type Yule process with mutations as $t\to\infty$. The random total number of leaves $\sum_{i=1}^k N_i(t)$ grows as well, so we need to consider the fraction of different types of cherries and pendants. We start with results in case the birth $\{q_i^{j_1j_2}\}_{i,j_1\le j_2\in\{1,\ldots,k\}}$ and mutation $\{q_i^j\}_{i\neq j\in\{1,\ldots,k\}}$ rates in the process are constant (time independent) and then generalize to the time varying case.

The matrix $\bm B(t)$ from Lemma~\ref{lem:meannudependant-eq} has nonnegative entries, except possibly for those on the diagonal. By the Perron-Frobenius theorem, if it is irreducible, there exists a largest simple eigenvalue $\lambda(t)$ of $\bm B(t)$ with right and left eigenvectors $\bm u(t), \bm v(t)$, respectively. We can assume that $\bm 1\cdot \bm u(t)=1$. In case that $\bm B(t)\equiv\bm B$ is time independent, we have the following result. Recall $\rho(t)$ is the mean total number of leaves in the tree.

\begin{lem}
\label{lem:limit-eta-tind}
If $\bm B(t)\equiv\bm B$ and is irreducible,  then $\bm \eta_\ell(t):= \displaystyle {\rho(t)}^{-1}{\bm \mu_\ell(t)}$, if it converges, satisfies
$$\lim_{t\rightarrow\infty} \bm \eta_\ell(t)=-u_\ell(\bm A_\ell-\lambda \bm I)^{-1} \bm {q}_{(\ell)},$$
where $\lambda$ is the largest real eigenvalue of $\bm B$ with coresponding right eigenvector $\bm u=(u_1,\ldots,u_k)$.
Furthermore, if $\bm \eta(t)=(\bm \eta_1(t),\ldots,\bm\eta_k(t))$, then $\displaystyle \bm \eta^{\star}(t):= {\rho(t)}^{-1}{\bm \gamma(t)}$, if it converges, satisfies
$$\lim_{t\rightarrow\infty} \bm \eta^\star(t)=-(\bm C-\lambda \bm I)^{-1}\bm U\lim_{t\rightarrow\infty}\bm \eta(t).$$
\end{lem}

\begin{proof}
Using $\bm \mu_\ell(t)=\rho(t)\bm \eta_\ell(t)$ in the differential equation for $\mu_\ell(t)$ from Proposition~\ref{prop:meanmudependant} we get
$$\frac{d\bm \eta_\ell(t)}{dt}=\big(\bm A_\ell-{\rho(t)}^{-1}\frac{d\rho(t)}{dt} \bm I\big)\bm \eta_\ell(t)+ {\rho(t)}^{-1}\bm {q}_{(\ell)}(t) \nu_\ell(t).$$
Assuming that $\lim_{t\rightarrow\infty}\bm \eta_\ell(t)$ exists, 
taking limit as $t\rightarrow\infty$ on both sides and using the fact that $\bm \eta_\ell(t)$ is continuous, we get
\begin{equation}\label{eq:differentialeta}
0=\Big(\bm A_\ell-\lim_{t\rightarrow\infty}{\rho(t)}^{-1}\frac{d\rho(t)}{dt} \bm I\Big)\lim_{t\rightarrow\infty}\bm \eta_\ell(t)+\bm {q}_{(\ell)}\lim_{t\rightarrow\infty}{\rho(t)}^{-1} \nu_\ell(t).
\end{equation}

Let $\bm J$ denote the Jordan representation form of the matrix $\bm B$, so that $\bm B=\bm P\bm J\bm P^{-1}$ and $\exp\{\bm B\}=\bm P\exp\{\bm J\}\bm P^{-1}$. By Corollary \ref{cor:meanrho},
\begin{eqnarray*}
\lim_{t\rightarrow\infty}{\rho(t)}^{-1}\frac{d\rho(t)}{dt}&=&\lim_{t\rightarrow\infty} \frac{\bm 1^\mathsf{T} \bm B\exp\{\bm Bt\}\bm \nu(0)}{\bm 1^\mathsf{T}\exp\{\bm Bt\}\bm \nu(0)}\\
&=&\frac{\bm 1^\mathsf{T}\lambda \bm u\bm v^\mathsf{T}\bm \nu(0)}{\bm 1^\mathsf{T} \bm u\bm v^\mathsf{T}\bm \nu(0)}=\lambda.
\end{eqnarray*}
Similarly, using Lemma~\ref{lem:meannudependant-eq}, 
\begin{eqnarray*}
\lim_{t\rightarrow\infty}{\rho(t)}^{-1}{\nu_\ell(t)}&=&\frac{\bm e_\ell\bm u\bm v^\mathsf{T}\bm e_a}{\bm 1^\mathsf{T} \bm u\bm v^\mathsf{T}\bm e_a}=u_\ell.
\end{eqnarray*}
We claim that $(\bm A_\ell-\lambda \bm I)$ is invertible. This is true because $\bm A_\ell$ is diagonally dominant by columns (see Remark \ref{rem:Adiagdominant}), and $\lambda\geq0$ ($\rho(t)$ is positive and increasing), which means that $(\bm A_\ell-\lambda \bm I)$ is diagonally dominant by columns as well. Hence, from (\ref{eq:differentialeta}) we get
$$\lim_{t\rightarrow\infty} \bm \eta_\ell(t)=-u_\ell(\bm A_\ell-\lambda \bm I)^{-1} \bm {q}_{(\ell)}.$$
The proof for $\lim_{t\rightarrow\infty} \bm\eta^\star(t)$ follows in the analogous steps, using $\bm \gamma(t)=\rho(t)\bm \eta^\star(t)$, the differential equation for $\bm \gamma(t)$ from Proposition~\ref{prop:meangamma} and the fact that $\bm C-\lambda \bm I$ is also diagonally dominant.
\end{proof}

To obtain a version of this result in the time varying case we need to make some assumptions on the behaviour of birth and mutation  rates in the long time limit.

\begin{thm}
\label{thm:limit-eta-tdepnotcomm}
If all the birth rates and mutation rates in the long term converge to 
limits $\{\lim_{t\rightarrow\infty} q_{i}^{j_1j_2}(t)\}_{i,j_1\le j_2\in\{1,\ldots,k\}}$ and $\{\lim_{t\rightarrow\infty}  q_{i}^{j}(t)\}_{i\neq j\in\{1,\ldots,k\}}$ 
such that the matrix $\lim_{t\rightarrow\infty} \bm B(t)$ is irreducible with maximum eigenvalue $\lambda$ and corresponding right and left eigenvectors $\bm u$ and $\bm v$ respectively;
then, assuming the limits below exist, 
$$\bm w_\ell:=\lim_{t\rightarrow\infty} \bm \eta_\ell(t)=-u_\ell \lim_{t\rightarrow\infty}(\bm A_\ell(t)-\lambda \bm I)^{-1} \lim_{t\rightarrow\infty}\bm {q}_{(\ell)}(t),$$
and
$$\bm w^\star:=\lim_{t\rightarrow\infty} \bm \eta^\star(t)=-\lim_{t\rightarrow\infty}(\bm C(t)-\lambda \bm I)^{-1}\bm U(t)\lim_{t\rightarrow\infty}\bm \eta(t),$$
\end{thm}

\begin{proof} The proof is similar to that in the constant rate case. 
Replacing $\bm \mu_\ell(t)=\rho(t)\bm \eta_\ell(t)$ in the differential equation  for $\mu_\ell(t)$ from Proposition~\ref{prop:meanmudependant},  we get
$$
\frac{d\bm \eta_\ell(t)}{dt}=\big(\bm A_\ell(t)-{\rho(t)}^{-1}\frac{d \rho(t)}{dt} \bm I\big)\bm \eta_\ell(t)+ {\rho(t)}^{-1}{\bm {q}_{(\ell)}(t) \nu_\ell(t)}
$$
Since $\lim_{t\rightarrow\infty}\bm \eta_\ell(t)$ exists, taking $t\rightarrow\infty$ on both sides, we get
\begin{equation}\label{eq:differentialetadependant}
0=\Big(\lim_{t\rightarrow\infty} \bm A_\ell(t)-\lim_{t\rightarrow\infty}{\rho(t)}^{-1}\frac{d \rho(t)}{dt}\bm I\Big)\lim_{t\rightarrow\infty}\bm \eta_\ell(t)+\lim_{t\rightarrow\infty}\bm {q}_{(\ell)}(t)\lim_{t\rightarrow\infty}{\rho(t)}^{-1}{ \nu_\ell(t)}
\end{equation}
From Lemma~\ref{lem:meannudependant-eq} we have $\frac{d\bm{\nu}(t)}{dt}=\bm B(t)\bm{\nu}(t)$, and defining $\bm{\beta}(t):={\rho(t)}^{-1}{\bm{\nu}(t)}$, we have 
$$
\frac{d\bm{ \beta}(t)}{dt}=\bm B(t) \bm{\beta}(t)-{\rho(t)}^{-1}\frac{d\rho(t)}{dt}\bm{\beta}(t),$$
which taking ${t\rightarrow\infty}$  on both sides gives
$$\lim_{t\rightarrow\infty}{\rho(t)}^{-1}\frac{d \rho(t)}{dt}\,\lim_{t\rightarrow\infty}\bm{ \beta}(t)=\lim_{t\rightarrow\infty} \bm B(t)\,\lim_{t\rightarrow\infty}\bm{\beta}(t).$$
By assumption the matrix $\lim_{t\rightarrow\infty} \bm B(t)$ has all finite entries and is irreducible, hence the vector $\lim_{t\rightarrow\infty}\bm{ \beta}(t)$ only has positive entries and the Perron-Frobenious Theorem implies that this vector is the eigenvector $\bm u$ and that $\lambda = \lim_{t\rightarrow\infty}{\rho(t)}^{-1}\frac{d\rho(t)}{dt}$.

We now claim that $(\bm A_\ell(t)-\lambda \bm I)$ is invertible. This is true because $\bm A_\ell(t)$ is diagonally dominant by columns (again see Remark \ref{rem:Adiagdominant}), and $\lambda \geq0$ (since $\rho(t)$ is positive and increasing), which means that $(\bm A_\ell(t)-\lambda \bm I)$ is diagonally dominant by columns as well. Hence, (\ref{eq:differentialetadependant}) implies
$$\lim_{t\rightarrow\infty} \bm \eta_\ell(t)=-u_\ell\lim_{t\rightarrow\infty}(\bm A_\ell(t)-\lambda \bm I)^{-1} \lim_{t\rightarrow\infty}\bm {q}_{(\ell)}(t)$$
as claimed.

The proof for $\lim_{t\rightarrow\infty} \bm\eta^\star(t)$ follows in the analogous steps, replacing $\bm \gamma(t)=\rho(t)\bm \eta^\star(t)$ in the differential equation for $\bm \gamma(t)$ from Proposition~\ref{prop:meangamma} and using the fact that $\bm C-\lambda \bm I$ is also diagonally dominant.
\end{proof}

\begin{rem}
In the special case that $\forall t\ge 0$ the matrices $\bm B(t)$ are irreducible, mutually diagonalizable, to matrices $\bm D(t)$, and have the same right and left eigenvectors $\bm u,\bm v$ for their corresponding maximum eigenvalues $\lambda(t)$, we can give a shorter proof: from Corollary~\ref{cor:meanrho};
\begin{eqnarray*}
\lim_{t\rightarrow\infty}\frac{\frac{d\rho(t)}{dt}}{\rho(t)}\!\!\!\!\!&=&\!\!\!\!\lim_{t\rightarrow\infty} \frac{\bm 1^\mathsf{T}\bm B(t)\exp\{\int_0^t \bm B(\tau) d\tau\}\bm \nu(0)}{\bm 1^\mathsf{T}\exp\{\int_0^t \bm B(\tau) d\tau\}\bm  \nu(0)}
=\lim_{t\rightarrow\infty} \frac{\bm 1^\mathsf{T}\bm P\bm D(t)\bm P^{-1}\bm P\exp\{\int_0^t \bm D(\tau) d\tau\}\bm P^{-1}\bm  \nu(0)}{\bm 1^\mathsf{T}\bm P\exp\{\int_0^t \bm D(\tau) d\tau\}\bm P^{-1}\bm  \nu(0)}\\
&=&\!\!\!\lim_{t\rightarrow\infty} \frac{\bm 1^\mathsf{T}\bm P\bm D(t)\exp\{\int_0^t \bm D(\tau) d\tau\}\bm P^{-1}\bm  \nu(0)}{\bm 1^\mathsf{T}\bm P\exp\{\int_0^t \bm D(\tau) d\tau\}\bm P^{-1}\bm  \nu(0)}
=\lim_{t\rightarrow\infty}\lambda(t)\frac{\bm 1^\mathsf{T}\bm u\bm v^\mathsf{T}\bm  \nu(0)}{\bm 1^\mathsf{T} \bm u\bm v^\mathsf{T}\bm  \nu(0)}=\lim_{t\rightarrow\infty}\lambda(t),
\end{eqnarray*}
 since the dominating terms are only those involving $ e^{\int_0^t\lambda(\tau)d\tau}$
with $ \bm u,\, \bm v$ as right and left eigenvectors of $\bm B(t)$ respectively; also,
\begin{eqnarray*}
\lim_{t\rightarrow\infty}\frac{ \nu_\ell(t)}{\rho(t)}=\frac{\bm e_\ell\bm u\bm v^\mathsf{T}\bm e_a}{\bm 1^\mathsf{T} \bm u\bm v^\mathsf{T}\bm e_a}=u_\ell,
\end{eqnarray*}
and substituting these in (\ref{eq:differentialetadependant}) gives the desired result.
\end{rem}

The asymptotic results allow one to infer the birth and mutation rate parameters of the models based on the number of cherries and pendants. Note that in the constant rate case, we have $k^2(k+1)/2+ k(k-1)$ parameters, and we have $k^2(k+1)/2+ k^2$ statistics which satisfy the relation: $2\sum_{\substack{\ell, i\leq j}}\eta_{\ell}^{ij}(t)+\sum_{i\neq j}{\eta}_i^{j}(t)=1$. One nonetheless needs some form of additional information in order to infer the model parameters, as  in the following result. Let $$r_{i}(t):=\sum_{j_1\leq j_2} q_{i}^{j_1,j_2}(t), \;\;i\in\{1,\ldots,k\}$$ denote the overall birth rates for each type.

\begin{cor}\label{cor:contsinferrence}
If the long term birth rates $\lim_{t\rightarrow\infty}r_{i}(t)$ and the maximum real eigenvalue $\lambda$ of $\lim_{t\rightarrow\infty}\bm B(t)$ are known, then the limits of the birth and mutation rates can be expressed in terms of the limiting fractions of cherries and pendants $\bm w=\lim_{t\rightarrow\infty} \bm \eta(t)$ and $\bm w^\star=\lim_{t\rightarrow\infty} \bm \eta^\star(t)$, where $\bm w_\ell=[w_\ell^{11}, \ldots w_\ell^{kk}]^{\rm T}$, $\forall \ell\in\{1,\ldots,k\}$ and $\bm w^\star=[w_1^1, \ldots, w_k^k]$.
\end{cor}

\begin{proof}
Observe that we can express $\bm C$ and $\bm U$ in terms of $r_i(t)$ as
$$[\bm C(t)]_{\ell m,ij}=\left\{\begin{array}{cl} -r_{m}(t)-\displaystyle\sum_{i\neq m}q_{m}^i(t)& \mbox { when } (\ell,m)=(i,j).\\    q_{j}^m(t)& \mbox { when } \ell=i,\, m\neq j.\\ \\0 &\mbox{ otherwise. }\end{array}\right.$$ and,
$$[\bm U(t)]_{\ell m,\ell' i j}=\left\{\begin{array}{cl}\displaystyle 2r_{m}(t)& \mbox { when } \ell=\ell',\, m=i=j.\\  \\ r_{i}(t)& \mbox { when } \ell=\ell',\, m=j> i.\\  \\  r_{j}(t)& \mbox { when } \ell=\ell',\, m=i<j.\\ \\ 0 &\mbox{ otherwise}.\end{array}\right.$$ 

Since $\lambda$ is known, Theorem~\ref{thm:limit-eta-tdepnotcomm} implies we have 
$$\lim_{t\rightarrow\infty}(\bm C(t)-\lambda \bm I)\bm w^\star +\lim_{t\rightarrow\infty}\bm U(t)\bm w=\bm 0$$
a linear system which, knowing the values of  $\lim_{t\rightarrow\infty}r_{i}(t)$ and $\lambda$, and given the values of $\bm w$ and $\bm w^\star$ from statistics of cherries and pendants,  depends only on the limits of the mutation rates $\lim_{t\rightarrow\infty}q_{i}^j(t)$, $i\neq j\in\{1,\ldots,k\}$. 

For each solution of this system in terms of the limiting mutation rates, we will have the values of $\lim_{t\rightarrow}q_{(\ell)}(t)$ for $\ell\in\{1,\ldots,k\}$ which can subsequently be used  in each of the systems
$$\lim_{t\rightarrow\infty}(\bm A_\ell(t)-\lambda \bm I)\bm w_\ell + u_\ell \lim_{t\rightarrow\infty}\bm{ q}_{(\ell)}(t)=\bm 0,\;\; \forall \ell\in\{1,\ldots,k\},$$
which is in fact a linear system in the branching rates because 
$$u_\ell=2\sum_{i}w_i^{\ell\ell}+\sum_{\substack{i, j<\ell}}w_i^{j\ell }+\sum_{\substack{i, j>\ell}}w_i^{\ell j}+\sum_{i\neq \ell} w_i^{\ell}.$$ 
It is therefore possible to get solutions of this system in terms of the limiting birth rates by expressing in terms of vectors $\bm w, \bm w^*$ as claimed.
\end{proof}

\begin{rem}\label{rem:reversibility-type-independent}
In the special case  that the overall birth rates $r_i(t)\equiv r_i\, \forall t $ are constants and $r_i\equiv r\, \forall i\in\{1,\ldots,k\}$  are independent of type, the maximum eigenvalue of $\bm B$ is simply $\lambda=r$, so in order to infer the birth and mutation rates we only need to know the overall growth rate $r$ and the statistics on the fractions of cherries and pendants. \end{rem}

\

\subsection{Some special cases of multi-type Yule models}

To illustrate how the asymptotic fractions of cherries and pendants can be used to infer the birth $\{q_i^{j_1j_2}\}_{i,j_1\le j_2\in\{1,\ldots,k\}}$ and mutation $\{q_i^j\}{i\neq j\in\{1,\ldots,k\}}$ rates in the model we consider two particular cases of the {\it `symmetric change of type'} models with $k=2$. We will assume that the two overall birth rates $r_i=\sum_{1\le j_1\le j_2\le 2}q_i^{j_1j_2}$ are independent of the type $r_1=r_2 =: r$, that birth rates are symmetric in parent type \,$\{q_1^{11}=q_2^{22}, q_1^{12}=q_2^{12}\}$, and that the same holds for mutation rates\, $\{q_1^2=q_2^1\}$.  We consider the following two such models:

(a) `cladogenetic change' model in which change in type can only occur at birth events and occurs independently for the offspring and parent: $$q_1^2=q_2^1=0,\quad q_1^{11}=q_2^{22}=r\,(1-p)^2,\, q_1^{12}=q_2^{12}=r\, 2p(1-p),\, q_1^{22}=q_2^{11}= rp^2$$ where $p\in(0,1)$ is the probability of type change at a birth event; 

(b) `anagenetic change' model in which change in type can only occur along the lineage: 
$$q_1^{11}=q_2^{22}=r,\, q_1^{12}=q_1^{22}=q_2^{12}=q_2^{11}=0, \quad q_1^2=q_2^1=rp$$ where $p$ is the relative rate of mutation along a lineage.\\

\noindent (a) Since in the cladogenetic case all mutation rates are zero, by Corollary~\ref{cor:contsinferrence} we only need to solve the system of equations $\{(\bm A_\ell-\lambda\bm I)\bm w_\ell + u_\ell \bm{q}_{(\ell)}=\bm 0\}_{ \ell=1,2}$ for the rates $\bm q_{(1)}, \bm q_{(2)}$. We have that  the matrix $\bm B$ is 
$$\bm B=\left[\begin{array}{cc} q_1^{11}-q_1^{22}-q_1^2 & 2q_2^{11}+q_2^{12}+q_2^1\\ 2q_1^{22}+q_1^{12}+q_1^2 & q_2^{22}-q_2^{11}-q_2^1 \end{array}\right]=r\left[\begin{array}{cc} 1-2p & 2p\\ 2p& 1-2p\end{array}\right],$$
with eigenvalue $\lambda=r$ and corresponding right eigenvector $u=[1/2,1/2]^{\rm T}$.
We have $q_1=q_2=r$ and the matrices $\bm A_1, \bm A_2$ are
$$\bm A_1=\bm A_2=\left[\begin{array}{ccc}-(q_1+q_1)&0&0\\0&-(q_1+q_2)&0\\0&0&-(q_2+q_2)\end{array}\right]=\left[\begin{array}{ccc}-2r&0&0\\0&-2r&0\\0&0&-2r\end{array}\right].$$
Solving the above system of equations for $q_{(1)}, q_{(2)}$ in terms of the asymptotic fractions of cherries and pendants $\bm w_\ell=[w_\ell^{11},w_\ell^{12},w_\ell^{22}]$ for $\ell\in\{1,2\}$ and $\bm w^\star=[w_1^1,w_1^2,w_2^1,w_2^2]$ gives $q_{(\ell)}=-2(\bm A_\ell-r\bm I)\bm w_\ell=-3\bm A_\ell\bm w_\ell=6r[w_\ell^{11},w_\ell^{12},w_\ell^{22}]^{\rm T}$, for $\ell=1,2$. Birth rates then are $$q_\ell^{j_1j_2}=6r w_\ell^{j_1j_2}\;\;\mbox{ for }\, \ell, j_1\le j_2\in\{1,2\}$$
This implies that the asymptotic fractions of cherries together with $p$ satisfy
$$p=1-\sqrt{6w_1^{11}}=\sqrt{6w_1^{22}}=\frac{1}{2}(1\pm\sqrt{1-12w_1^{12}}).$$
Note that if we ignore edge lengths in this tree, we essentially get the random discrete tree arising from the symmetric Markov propagation model, briefly discussed at the end of Subsection~\ref{subsec:ERMlimits}, in which the propagation matrix $\bm S$ is symmetric with $s_{12}=s_{21}=p,\; s_{11}=s_{22}=1-p$.\\

\noindent (b) In the anagenetic case, by Corollary~\ref{cor:contsinferrence} we need to solve the system of equations $(\bm C-\lambda\bm I)\bm w^\star+\bm U\bm w=\bm 0$ for the mutation rates $q_1^2=q_2^1=rp$. The matrix $\bm B$ is 
$$\bm B=\left[\begin{array}{cc} q_1^{11}-q_1^{22}-q_1^2 & 2q_2^{11}+q_2^{12}+q_2^1\\ 2q_1^{22}+q_1^{12}+q_1^2 & q_2^{22}-q_2^{11}-q_2^1 \end{array}\right]=r\left[\begin{array}{cc} 1-p&p  \\ p &1-p\end{array}\right],$$
with eigenvalue $\lambda=r$ and corresponding right eigenvector $u=[1/2,1/2]^{\rm T}$.
Also $q_1=q_2=r+rp$, the matrix $\bm C-\lambda\bm I=\bm C-r\bm I$ is 
$$\bm C-r\bm I=\left[\begin{array}{cccc}-q_1-r&q_2^1&0&0\\q_1^2&-q_2-r&0&0\\0&0&-q_1-r&q_2^1\\0&0&q_1^2&-q_2-r \end{array}\right]=r\left[\begin{array}{cccc}-(2+p)&p&0&0\\p&-(2+p)&0&0\\0&0&-(2+p)&p\\0&0&p&-(2+p) \end{array}\right]$$
and the matrix $\bm U$ is
$$\bm U=\left[\begin{array}{cccccc}2q_1^{11}&q_2^{22}&0&0&0&0\\0&q_1^{11}&2q_2^{22}&0&0&0\\0&0&0&2q_1^{11}&q_2^{22}&0\\0&0&0&0&q_1^{11}&2q_2^{22} \end{array}\right]
=\left[\begin{array}{cccccc}2r&r&0&0&0&0\\0&r&2r&0&0&0\\0&0&0&2r&r&0\\0&0&0&0&r&2r \end{array}\right]$$
Solving the above system for $q_1^2=q_2^1=rp$ in terms of the asymptotic fractions of cherries and pendants $\bm w_\ell=[w_\ell^{11},w_\ell^{12},w_\ell^{22}]$ for $\ell\in\{1,2\}$ and $\bm w^\star=[w_1^1,w_1^2,w_2^1,w_2^2]$ implies that the asymptotic fractions of cherries and pendants as well as $p$ satisfy
$$p=\frac{2w_1^{11}+w_1^{12}-2w_1^1}{w_1^1-w_1^2}=\frac{w_1^{12}+2w_1^{22}-2w_1^2}{w_1^2-w_1^1}=\frac{2w_2^{22}+w_2^{12}-2w_2^1}{w_2^1-w_2^2}=\frac{w_2^{12}+2w_2^{22}-2w_2^2}{w_2^2-w_2^1}.$$
The matrices $\bm A_1, \bm A_2$ are
$$\bm A_1=\bm A_2=\left[\begin{array}{ccc}-(q_1+q_1)&0&0\\0&-(q_1+q_2)&0\\0&0&-(q_2+q_2)\end{array}\right]=r\left[\begin{array}{ccc}-2(1+p)&0&0\\0&-2(1+p)&0\\0&0&-2(1+p)\end{array}\right]$$
and the value of $p$ should make the system of equations $\{(\bm A_\ell-r\bm I)\bm w_\ell + \frac12 \bm{q}_{(\ell)}=\bm 0\}_{ \ell=1,2}$ with $\bm q_{(1)}=[r,0,0]^{\rm T}, \bm q_{(2)}=[0,0,r]^{\rm T}$ a consistent one.
With $p$ as above, birth and mutation rates then are 
$$q_1^{11}=q_2^{22}=r, \, q_1^{12}=q_1^{22}=q_2^{12}=q_2^{11}=0, \quad q_1^2=q_2^1=rp.$$

Our results can be used together with what is previously known about predictive accuracy of a reconstruction method, such as maximum parsimony, majority rule and maximum likelihood, for the ancestral states. Predictive accuracy is measured in terms of the expected value (over all sample trees in the random model) of the probability that the predicted type of the root is correct. There are a number of known results (\cite{GascSteel}, \cite{MossSteel2}) on when a  reconstruction method for the type of the root in the tree is more accurate than a uniform guess on its value. 
For the above models of symmetric change of type (with $k=2$) the results of \cite{GascSteel} state that the predictive accuracy of the maximum parsimony method is asymptotically 1/2 iff $r\le 6s$; and the predictive accuracy of any method is asymptotically 1/2 if $r\le 4s$; where $s=r\, p$ denotes the substitution rate in this symmetric propagation model. Results of \cite{MossSteel2} state that majority rule is more accurate than a uniformly random guess iff $r>4s$. Our expressions for $p=s/r$ allow one to approximately determine whether in a given tree the type of the root can be accurately predicted by one of these methods or not.\\

Unfortunately, inference of birth and mutation rates cannot be used in the {\it `asymmetric change of type'} models, such as:

(c) cladogenetic change with  $$q_1^2=q_2^1=0,\quad q_1^{11}=r, q_1^{12}=q_1^{22}=0,\, q_2^{22}=r\,(1-p)^2, q_2^{12}=r\, 2p(1-p), q_2^{11}= rp^2$$ 
here the matrix $\bm B=r\left[\begin{array}{cc} 1 &2p\\ 0& 1-2p\end{array}\\\right]$ still has maximal eigenvalue $\lambda=r$ but is reducible; and

(d) anagenetic change with
$$q_1^{11}=q_2^{22}=r,\, q_1^{12}=q_1^{22}=q_2^{12}=q_2^{11}=0, \quad q_1^2=0, \, q_2^1=rp$$ 
where the matrix $\bm B=r\left[\begin{array}{cc} 1 &p\\0 &1-p\end{array}\right]$ is reducible as well.

\

\subsection{Comparison for numbers of cherries in different models}

Consider a general multi-type Yule tree on $k=2$ types but without mutations. Its overall birth rates $q_1(t)=\sum_{j_1\le j_2}q_1^{j_1j_2}(t)$ and $q_2(t)=\sum_{j_1\le j_2}q_2^{j_1j_2}(t)$ are generally not the same, which implies that the probabilities at which lineages of each type are chosen to be the next one to give birth are not the same (there is `non-neutrality' in types). Let $a_1(t):=q_1(t)/(q_1(t)+q_2(t))$ and $a_2(t):=q_2(t)/(q_1(t)+q_2(t))=1-a_1(t)$ denote the weights proportional which lineages of types 1 and 2, respectively, get chosen to give birth (see Remark~\ref{rem:fromYuletoERM}. For any two such models $\{q_i^{j_1j_2}(t)\}_{i,j_1\le j_1\in\{1,2\}}$ and $\{{q'}_i^{j_1j_2}(t)\}_{i,j_1\le j_1\in\{1,2\}}$ we can compare the weights $a_1$ and $a'_1$ of choosing type $1$ lineages. We provide a comparison between the asymptotic fraction of different types of cherries $\bm w_1=[w_1^{11}, w_1^{12}, w_1^{22}]^{\rm T}$ and  $\bm w_2=[w_2^{11}, w_2^{12}, w_2^{22}]^{\rm T}$ in the two models based on the comparison of their weights $a_1(t)$ and $a_2(t)=1-a_1(t)$ of choosing a lineage of different types to give birth. 

\begin{prop}\label{prop:comparison} Assume that the birth rates $\{q_i^{j_1j_2}(t)\}_{i,j_1\le j_1\in\{1,2\}}$ and $\{{q'}_i^{j_1j_2}(t)\}_{i,j_1\le j_1\in\{1,2\}}$ in the two models are such that, their limits $q_\ell^{j_1j_2}:=\lim_{t\rightarrow\infty}q_\ell^{j_1j_2}(t), q_\ell:=\lim_{t\rightarrow\infty}q_\ell(t)$ satisfy
\begin{equation}\label{condition-splitting-prob}
\frac{q_1^{11}-q_1^{22}}{q_1}=1+\frac{q_2^{11}-q_2^{22}}{q_2}, \quad \frac{{q'}_1^{11}-{q'}_1^{22}}{{q'}_1}=1+\frac{{q'}_2^{11}-{q'}_2^{22}}{{q'}_2}
\end{equation}
Then, the asymptotic proportions of cherries of type 1 and type 2 in the two models satisfy  monotonicity in terms of weights $a_1$ and $a'_1$ given by
$$a_1<a'_1 \quad\Rightarrow\quad w_1^{j_1j_2}<{w'}_1^{j_1j_2},\,\forall j_1\le j_2 \; \mbox{ and }\;w_2^{j_1j_2}>{w'}_2^{j_1j_2},\,\forall j_1\le j_2$$
where $a_1:=\lim_{t\rightarrow\infty}a_1(t), a'_1:=\lim_{t\rightarrow\infty}a'_1(t)$ denote the limiting weights of type 1 lineages.
\end{prop}

\begin{proof}
Theorem~\ref{thm:limit-eta-tdepnotcomm} implies that the vectors $\bm w_\ell=[w_\ell^{11},w_\ell^{12},w_\ell^{22}]$ for $\ell\in\{1,2\}$ satisfy 
$$\bm w_\ell=\lim_{t\rightarrow\infty} \bm \eta_\ell(t)=-u_\ell\lim_{t\rightarrow\infty}(\bm A_\ell(t)-\lambda \bm I)^{-1} \lim_{t\rightarrow\infty}\bm q_{(\ell)}.$$
Since we are considering multi-type Yule models with mutation rates $q_1^2(t)=q_2^1(t)=0$, the matrix $\bm A(t)$ from Proposition~\ref{prop:meanmudependant} depends only on $q_1(t)$ and $q_2(t)$. 

Let $p_\ell^{j_1j_2}:=q_\ell^{j_1j_2}/q_\ell$ denote the probabilities that a birth event at a lineage of type $\ell$ results in types $j_1, j_2$ (see Remark~\ref{rem:fromYuletoERM}) in the limit as $t\rightarrow\infty$. Then 
$q_\ell^{j_1j_2}=p_\ell^{j_1j_2}a_\ell(q_1+q_2)$ and the assumption on the limiting birth rates becomes
$$p_1^{11}+p_2^{22}=1+p_1^{22}+p_2^{11}$$
Substituting all this into the equations for $\bm w_\ell$ above, and using $a_2=1-a_1$, we obtain expressions for $\bm w_1,\bm w_2$ that are written entirely in terms of probabilities $p_\ell^{j_1j_2}$ and weight $a_1$:
$$w_1^{11}=\frac{a_1p_1^{11}(p_1^{11}-p_1^{22})}{2p_1^{11}a_1+a_1-2p_1^{22}a_1-p_1^{11}+p_1^{22}+1},$$
$$w_1^{12}=\frac{a_1p_1^{12}(p_1^{11}-p_1^{22})}{2p_1^{11}a_1-a_1-2p_1^{22}a_1-p_1^{11}+p_1^{22}+2},$$
$$w_1^{22}=\frac{a_1p_1^{22}(p_1^{11}-p_1^{22})}{2p_1^{11}a_1-3a_1-2p_1^{22}a_1-p_1^{11}+p_1^{22}+3},$$
$$w_2^{11}=\frac{p_2^{11}(1-p_1^{11}+p_1^{22})(1-a_1)}{2p_1^{11}a_1+a_1-2p_1^{22}a_1-p_1^{11}+p_1^{22}+1},$$
$$w_2^{12}=\frac{p_2^{12}(1-p_1^{11}+p_1^{22})(1-a_1)}{2p_1^{11}a_1-a_1-2p_1^{22}a_1-p_1^{11}+p_1^{22}+2},$$
$$w_2^{11}=\frac{p_2^{22}(1-p_1^{11}+p_1^{22})(1-a_1)}{2p_1^{11}a_1-3a_1-2p_1^{22}a_1-p_1^{11}+p_1^{22}+3}.$$
In order to prove the monotonicity of $w_\ell^{j_1j_2}$ as a function of $a_1$ for $\ell,j_1,j_2\in\{1,2\}$, it suffices to check their first derivate with respect to $a_1$:
$$\frac{\partial w_1^{11}}{\partial a_1}=\frac{p_1^{11}(p_1^{11}-p_1^{22})(1-p_1^{11}+p_1^{22})}{(2p_1^{11}a_1+a_1-2p_1^{22}a_1-p_1^{11}+p_1^{22}+1)^2}>0,$$
$$\frac{\partial w_1^{12}}{\partial a_1}=\frac{p_1^{12}(p_1^{11}-p_1^{22})(2-p_1^{11}+p_1^{22})}{(2p_1^{11}a_1-a_1-2p_1^{22}a_1-p_1^{11}+p_1^{22}+2)^2}>0,$$
$$\frac{\partial w_1^{22}}{\partial a_1}=\frac{p_1^{22}(p_1^{11}-p_1^{22})(3-p_1^{11}+p_1^{22})}{(2p_1^{11}a_1-3a_1-2p_1^{22}a_1-p_1^{11}+p_1^{22}+3)^2}>0,$$
$$\frac{\partial w_2^{11}}{\partial a_1}=\frac{-p_2^{11}(2-(p_1^{11}-p_1^{22})(1+p_1^{11}-p_1^{22}))}{(2p_1^{11}a_1+a_1-2p_1^{22}a_1-p_1^{11}+p_1^{22}+1)^2}<0,$$
$$\frac{\partial w_2^{12}}{\partial a_1}=\frac{-p_2^{12}(1-(p_1^{11}-p_1^{22})^2)}{(2p_1^{11}a_1-a_1-2p_1^{22}a_1-p_1^{11}+p_1^{22}+2)^2}<0,$$
$$\frac{\partial w_2^{22}}{\partial a_1}=\frac{-p_2^{22}(p_1^{11}-p_1^{22})(1-p_1^{11}+p_1^{22})}{(2p_1^{11}a_1-3a_1-2p_1^{22}a_1-p_1^{11}+p_1^{22}+3)^2}<0,$$
and the result follows.
\end{proof}

The assumption for the limiting birth rates can be satisfied in relevant models. For example, in a  process where at a birth event types of the two continuing lineages are assigned according to a Markov process with transition probabilities $(s_{ij})_{i,j\in\{1,2\}}$, the probabilities $p_\ell^{j_1j_2}$ are 
$$p_1^{11}=(1-s_{12})^2,\quad p_1^{12}=2(1-s_{12})s_{12},\quad p_1^{22}=(s_{12})^2,$$
$$p_2^{22}=(1-s_{21})^2,\quad p_2^{12}=2(1-s_{21})s_{21},\quad p_2^{11}=(s_{21})^2,$$
and the assumption (\ref{condition-splitting-prob}) is equivalent to $s_{12}+s_{21}=1/2$. 

\begin{rem}
In the special case when the overall birth rates are equal  $a_1=a_2=1/2$ (`neutral' underlying tree shape), if edge lengths in the multi-type Yule tree without mutations are ignored the resulting distribution on the tree is that of a corresponding multi-type ERM model. Accordingly, 
 the asymptotic fractions of cherries we obtained in the proof of Proposition~\ref{prop:comparison} are in fact the same as asymptotic fractions obtained in Theorem~\ref{thm:polya-coro-1} for the multi-type ERM trees with probabilities $\{p_i^{j_1j_2}\}_{i,j_1\le j_2\in\{1,2\}}$. 
 \end{rem}



\begin{thebibliography}{99}		
\bibitem{Aldous1}
	\textsc{Aldous, D.J.}
	(1996)\\
	Probability distributions on cladograms.
	\emph{Random Discrete Structures}, (IMA Volumes Math.Appl. 76), 1-18.
\bibitem{Aldous2}
	\textsc{Aldous, D.J.}
	 (2001)\\
	 Stochastic Models and Descriptive Statistics for Phylogenetic Trees, from Yule to Today.
	 \emph{Stat.Sci.} 16(1): 23--34.
\bibitem{AldPopovic}	
	\textsc{Aldous, D., Popovic, L.}
	(2005)\\
	A critical branching process model for biodiversity. 
	\emph{Adv. Appl. Probab.} {\bf 37} 1094--1115. doi: 10.1239/aap/1134587755
\bibitem{Fitzjohn10}
	\textsc{Fitzjohn, R.G.}
	(2010)\\
	Quantitative traits and diversification. 
	\emph{Syst. Biol.} 59:619-633
\bibitem{FitzT}
	\textsc{Fitzjohn, R.G.}
	(2012)\\
	What drives biological diversification? detecting traits under species selection.
	University of British Columbia, PhD Thesis.
\bibitem{GascSteel}
	\textsc{Gascuel, O., Steel, M.}
	(2014)\\
	Predicting the ancestral character changes in a tree is typically easier than predicting the root state.
	\emph{Syst. Biol.} 63 (3), 421-435.
\bibitem{GoldIgic12}
	\textsc{Goldberg, E. E., Igic, B.}
	(2012)\\
	Tempo and mode in plant breeding system evolution. 
	\emph{Evolution} 66:3701-3709.
\bibitem{Goldberg11}
	\textsc{Goldberg, E. E., Lancaster, L.T., Ree, R.H.}
	(2011)\\
	Phylogenetic inference of reciprocal effects between geographic range evolution and diversification.
	\emph{Syst. Biol.} 60:451-465
\bibitem{Harding}
	\textsc{Harding, E.F.}
	(1971)\\
	The probabilities of rooted tree-shapes generated by random bifurcation.
	\emph{Adv.Appl.Prob.} 3:44--77.
\bibitem{J04}
	\textsc{Janson, S.}
     	(2004)\\
	Functional limit theorems for multitype branching processes and generalized {P}\'olya urns.
 	\emph{Stochastic Processes and their Applications}, 110(2):177--245.
\bibitem{J11}	
	\textsc{Jones, G.}
	(2011)\\
	Calculations for multi-type age-dependent binary branching processes. 
	\emph{J. of Math. Biol.}, 63(1):33--56.
\bibitem{LambPopovic}
	\textsc{Lambert, A., Popovic, L.}
	(2013)\\
	The coalescent point-process of branching trees. 
	\emph{Ann. Appl. Prob.} 23(1):99--144.  doi: 10.1214/11-AAP820
\bibitem{Maddison07}
	\textsc{Maddison, W.P., Midford, P.E., Otto, S.P.}
	(2007)\\
	Estimating a binary character's effect on speciation and extinction.
	\emph{Syst. Biol.} 56.5: 701-710.
\bibitem{McKSteel}
	\textsc{McKenzie, A., Steel, M.} 
	 (2000)\\
	Distributions of cherries for two models of trees.
	\emph{Mathematical biosciences} 164(1):81--92.
\bibitem{M08}
	\textsc{Mode, C.J.} 
	(1962)\\
	Some multi-dimensional birth and death processes and their applications in population genetics. 
	\emph{International Biometric Society}, 18(4):543--567.
\bibitem{MooersHeard}	
	\textsc{Mooers, A.O., Heard, S.B.}
	(1997)\\
	Inferring evolutionary process from phylogenetic tree shape.
	\emph{Quarterly Review of Biology}:31--54.
\bibitem{MosselS1}
	\textsc{Mossel, E.}
	(2004)
	Phase transitions in phylogeny.
	\emph{Trans. Am. Math. Soc.} 356(6):2379-2404
\bibitem{MosselS2}
	\textsc{Mossel, E.}
	(2004)
	Survey-Information Flow on Trees.
	\emph{DIMACS series in discrete mathematics and theoretical computer science.} 63:155-170.
\bibitem{MossSteel1}
	\textsc{Mossel, E., Steel, M.}
	(2005)\\
	How much can evolved characters tell us about the tree that generated them?
	\emph{Mathematics of evolution and phylogeny}, 384-412.
\bibitem{MossSteel2}
	\textsc{Mossel, E., Steel, M.}
	(2014)\\
	 Majority rule has transition ration 4 on Yule trees under a 2-state symmetric model
	 \emph{J. of Theor. Biol.}.
\bibitem{Neeetal}
	\textsc{Nee, S., May, R.H., Harvey, P.H.} 
	(1994)\\
	The reconstructed evolutionary process.
	\emph{Phil. Trans. Roy. Soc. B}, 344(1309):305-311.
\bibitem{NGSmith}
	\textsc{NG, J, Smith, S.D.}
	(2014)\\
	How traits shape trees: new approaches for detecting character state-dependent lineage diversification.
	\emph{J. Evol. Bio.} (published online 25 jul). doi: 10.1111/jeb.12460
\bibitem{PopRivas}
	\textsc{Popovic, L., Rivas, M.}
	(2014)\\
	The coalescent point-process of multi-type branching trees. 
	\emph{Stoch. Proc. Appl.} 124(12):4120--4148.
\bibitem{S96}
	\textsc{Smythe, R.T.}
	(1996)\\
     	Central limit theorems for urn models,
  	\emph{Stochastic Processes and their Applications}, 65(1):115--137.
\bibitem{Yule}
	\textsc{Yule, G.U.} A mathematical theory of evolution, based on the conclusions of Dr. J. C. Willis
	(1924)\\
	\emph{Philos. Trans. Roy. Soc. London Ser. B}, 213:21--87.
\end{thebibliography}
\end{document}